\newif\ifarxiv
\newif\ifieeemanuscript
\newif\ifieeearticle
\newif\ifleftovers
\newif\ifdraft

\arxivtrue

\ifarxiv
  \documentclass[a4paper,12pt]{article}
\else\ifieeemanuscript
  \documentclass[lettersize,onecolumn]{IEEEtran}
\else\ifieeearticle
  \documentclass[lettersize,journal]{IEEEtran}
\fi\fi\fi

\ifarxiv
  \usepackage[textwidth=145mm, top=25mm, bottom=30mm]{geometry}
  \ifdraft
  \usepackage[paperwidth=150mm,textwidth=145mm, paperheight=200mm, top=5mm,bottom=15mm]{geometry}
  \fi
\else
  \usepackage{cite}
  \usepackage{textcomp}
  \usepackage{array}
  \usepackage[caption=false,font=normalsize,labelfont=sf,textfont=sf]{subfig}
  \usepackage{stfloats}
\fi

\usepackage[utf8]{inputenc}
\usepackage[colorlinks=true, allcolors=blue]{hyperref}
\usepackage{amsmath,amsfonts,amsthm,amssymb}
\usepackage{enumerate}
\usepackage{lslbasicmath}
\usepackage{lslcomm}
\usepackage{tablefootnote}
\usepackage{mathtools}

\theoremstyle{plain}
\newtheorem{theorem}{Theorem}[section]
\newtheorem{lemma}[theorem]{Lemma}
\newtheorem{proposition}[theorem]{Proposition}

\newtheorem*{theorem*}{Theorem}
\newtheorem*{lemma*}{Lemma}
\newtheorem*{proposition*}{Proposition}
\newtheorem*{conjecture*}{Conjecture}
\newtheorem{fact*}{Fact}

\theoremstyle{definition}

\newtheorem{remark}[theorem]{Remark}

\newtheorem*{definition*}{Definition}
\newtheorem*{question*}{Question}
\newtheorem*{example*}{Example}
\newtheorem*{remark*}{Remark}
\newtheorem*{remarks*}{Remarks}
\newtheorem*{exercise*}{Exercise}
\newtheorem*{assumption*}{Assumption}

\numberwithin{equation}{section}

\newcommand{\new}[1]{\textit{#1}}

\newcommand{\ev}{\operatorname{ev}}
\renewcommand{\law}{\mathcal{L}}

\newcommand{\cou}{\operatorname{cou}}
\newcommand{\cum}{\operatorname{agg}}

\newcommand{\Hel}{\operatorname{Hel}}
\newcommand{\Tsa}{T}
\newcommand{\Ren}{R}
\newcommand{\KL}{\operatorname{KL}}
\newcommand{\intlim}{\int\displaylimits}

\newcommand{\la}{\lambda}
\newcommand{\al}{\alpha}

\newcommand{\titlea}{Information divergences and likelihood ratios of Poisson processes and point patterns}

\begin{document}

\title{\titlea}
\author{Lasse Leskelä}
\date{\today}

\ifieeemanuscript
\fi

\maketitle

\begin{abstract}
This article develops an analytical framework for studying information divergences and likelihood ratios associated with Poisson processes and point patterns on general measurable spaces.  The main results include explicit analytical formulas for Kullback--Leibler divergences, \Renyi divergences, Hellinger distances, and likelihood ratios of the laws of Poisson point patterns in terms of their intensity measures.  The general results yield similar formulas for inhomogeneous Poisson processes, compound Poisson processes, as well as spatial and marked Poisson point patterns.  Additional results include simple characterisations of absolute continuity,  mutual singularity, and the existence of common dominating measures.  The analytical toolbox is based on Tsallis divergences of sigma-finite measures on abstract measurable spaces.  The treatment is purely information-theoretic and free of topological assumptions.
\end{abstract}

\newcommand{\LLkeywords}{Poisson random measure, inhomogeneous Poisson process, point process, spatial point pattern, \Renyi divergence, Tsallis divergence, Hellinger distance, mutual information, Chernoff information,
Bhattacharyya distance}

\ifarxiv
 {\bf Keywords:} \LLkeywords
\else
 \begin{IEEEkeywords} \LLkeywords \end{IEEEkeywords}
\fi

\ifarxiv
\tableofcontents
\fi

\section{Introduction}


A point pattern or a point process (PP) represents a countable set of points in space or time.
Poisson PPs  are fundamental statistical models
for generating randomly scattered points
on the real line, in a Euclidean space, or in an abstract measurable space $S$.
They are encountered in a wide range of applications such as
archeology \cite{Bevan_2020}, 
\cite{Picard_Reynaud-Bouret_Roquain_2018}, 
astronomy \cite{Snethlage_Martinez_Stoyan_Saar_2002},
forestry statistics \cite{Stoyan_Penttinen_2000,Kuronen_Leskela_2013},
machine learning \cite{Birrell_etal_2021,Abbe_Sandon_2015_Community,Abbe_Baccelli_Sankararaman_2021,Avrachenkov_Kumar_Leskela_2024+},
neuroscience and genomics \cite{Reynaud-Bouret_Schbath_2010,Picard_Reynaud-Bouret_Roquain_2018}, and
queueing systems \cite{Atar_Budhiraja_Dupuis_Wu_2021,Leskela_2022}.
The law of a Poisson PP is a probability measure $P_\la$ characterised by an intensity measure $\la$,
so that $\la(B)$ indicates the expected number of points in $B \subset S$.
In statistical inference, it is important to understand how the intensity measure can be identified
from data.  Statistical research devoted
to this question has a long history --- well summarised in standard textbooks \cite{Karr_1991,Daley_Vere-Jones_2003,Ilian_Penttinen_Stoyan_Stoyan_2008,Diggle_2013}.  The main analytical approaches
include computing and estimating likelihood ratios $\frac{dP_\la}{dP_\mu}$ and information divergences
of laws of Poisson PPs with intensity measures $\la$ and $\mu$.

Likelihood ratios are easy to compute for standard families of probability distributions on finite-dimensional spaces,
but not so for probability measures of infinite-dimensional objects such as paths of stochastic processes
or spaces of point patterns.
In fact, even verifying the absolute continuity of a pair
of probability measures, a necessary condition for the existence of a likelihood ratio, can be nontrivial.
For Poisson PPs with finite intensity measures, a classical result \cite{Karr_1991,Reiss_1993}
states
that
$P_\la \ll P_\mu$ if and only if $\la \ll \mu$, in which case a likelihood ratio is given by
\begin{equation}
 \label{eq:IntroductionLR}
 \frac{d P_\la}{d P_\mu}(\eta)
 \weq \exp\left( \int_S \log \phi \, d\eta + \int_S (1-\phi) \, d\mu \right)
\end{equation}
with $\phi = \frac{d\la}{d\mu}$ being a density of the intensity measures.
Using this formula, it is easy to compute various types of information divergences and distances for $P_\la$ and $P_\mu$.
%
For Poisson PPs with general sigma-finite intensity measures, the
description and even the existence of a likelihood ratio is far less obvious.
To see why, note that the rightmost integral in \eqref{eq:IntroductionLR}
equals
$\int_S (1-\phi) \, d\mu = \mu(S) - \la(S)$
for finite intensity measures, but for infinite intensity measures this integral might not exist.
For Poisson PPs with general sigma-finite intensity measures,
most of the known results 
\cite{Skorohod_1957,Liese_1975,Karr_1983,Takahashi_1990} are restricted to locally compact Polish
spaces, thereby ruling out e.g.\ infinite-dimensional Hilbert spaces.

\subsection{Main contributions}

This article develops a framework for computing likelihood ratios and information divergences
in the most general natural setting, for Poisson PPs with sigma-finite intensity measures on a general measurable space. The purely information-theoretic approach makes no topological assumptions,
and allows one to work with point patterns in high- and infinite-dimensional spaces without worrying about
topological regularity properties.
A key contribution is an explicit formula (Theorem~\ref{the:PoissonDensity}) for the likelihood ratio 
$\frac{dP_\la}{dP_\mu}$ that is applicable to all Poisson PP distributions with $P_\la \ll P_\mu$.
This result facilitates the derivation of a characterisation for pairs of Poisson PPs whose
laws are dominated by a Poisson PP distribution (Theorem~\ref{the:DominatingPPP}).

Furthermore, the article provides a comprehensive characterisation of R{\'e}nyi and Kullback--Leibler divergences of Poisson PPs (Theorems~\ref{the:PPPRenyi}--\ref{the:PoissonKL}), showing that these divergences can be expressed as generalised Tsallis divergences of associated intensity measures. It also extends the definition of Tsallis divergences from probability measures to sigma-finite measures, representing them as linear combinations of R{\'e}nyi divergences of Poisson distributions (Theorem~\ref{the:Tsallis}). These Poisson--R{\'e}nyi--Tsallis relationships yield a simplified characterisation for the absolute continuity and mutual singularity of general Poisson PP distributions.


The practical applicability of these results is demonstrated in various contexts, including Poisson processes, compound Poisson processes, marked Poisson point patterns, and Chernoff information of Poisson vectors.

\subsection{Outline}

The rest of the article is organised as follows.
Section~\ref{sec:Preliminaries} introduces notations and definitions.
Section~\ref{sec:Tsallis} develops theoretical foundations for Tsallis divergences of sigma-finite measures.
Section~\ref{sec:LR} presents the main results concerning likelihood ratios of Poisson PP distributions.
Section~\ref{sec:Divergences} presents the main results about information divergences of Poisson PPs.
Section~\ref{sec:Applications} illustrates how the main results can applied to analyse Poisson processes, compound Poisson processes, marked Poisson point patterns, and Chernoff information of Poisson vectors.
Section~\ref{sec:Proofs} contains the technical proofs of the main results,
and
Section~\ref{sec:Conclusions} concludes.

\section{Preliminaries}
\label{sec:Preliminaries}

\subsection{Measures}
\label{sec:Measures}
Standard conventions of measure theory and Lebesgue integration \cite{Kallenberg_2002} are used.
The sets of nonnegative integers and nonnegative real numbers are denoted by
$\Z_+$ and $\R_+$, respectively.
For measures $\la,\mu$ on a measurable space $(S,\cS)$, the notation $\la \ll \mu$ means that $\la$ is \new{absolutely continuous} with respect to $\mu$, that is, $\mu(A) = 0$ for all $A \in \cS$ such that $\la(A)=0$.
We denote $\la \perp \mu$ and say that $\la,\mu$ are mutually singular
if there exist a measurable set $B$ such that $\la(B^c)=0$ and $\mu(B) = 0$.
A measurable function $f \colon S \to \R_+$ is called a \new{density} (or Radon--Nikodym derivative) of $\la$ with respect to $\mu$ when $\mu(A) = \int_A f \, d\mu$ for all $A \in \cS$; in this case we denote $f = \frac{d\la}{d\mu}$.
A density of probability measure is called a \new{likelihood ratio}.

The symbol $\delta_x$ refers to the Dirac measure at $x$.
For a number $c \ge 0$, the symbol $\Poi(c)$ denotes the Poisson probability distribution with mean $c$ and density $k \mapsto e^{-c} \frac{c^k}{k!}$ with respect to the counting measure on $\Z_+$, with the standard conventions that $\Poi(0) = \delta_0$ and $\Poi(\infty) = \delta_{\infty}$.

\subsection{Point patterns}

A point pattern is a countable collection of points, possibly with multiplicities,
in a measurable space $(S,\cS)$. Such a collection is naturally represented as a measure $\eta$ on
$(S,\cS)$, so that $\eta(B)$ equals the number of points in $B \in \cS$.
The requirement on countability is guaranteed when $\eta = \sum_{n=1}^\infty \eta_n$ for
some finite measures $\eta_n$.  Following \cite{Last_Penrose_2018}, we define $N(S)$ as the set of all
measures that can be written as a countable sum of integer-valued finite measures on $(S,\cS)$,
and equip it with the sigma-algebra $\cN(S)$ generated by the evaluation maps $\ev_B \colon \eta \mapsto \eta(B)$, $B \in \cS$.  An element of $N(S)$ is called a \new{point pattern} or a \new{point process} (PP).

\subsection{Poisson PPs}

Poisson PPs are defined in the standard manner, see \cite{Kingman_1967,Kallenberg_2002,Last_Penrose_2018} for general background.  A \new{PP distribution} is a probability measure on $(N(S), \cN(S))$.
The \new{Laplace functional} of a PP distribution $P$ is the map
that assigns to every measurable function $u \colon S \to [0,\infty]$ a number
$L_P(u) = \int_{N(S)} \exp( - \int_X u \, d\eta ) \, P(d\eta) \in [0,1]$.
Given a sigma-finite measure $\la$ on $(S,\cS)$, the
\new{Poisson PP distribution} with intensity measure $\la$ is the unique probability measure $P_\la$
on $(N(S), \cN(S))$ such that
\begin{enumerate}[(i)]
\item $P_\la \circ (\ev_{B_1}, \dots, \ev_{B_n})^{-1} = \bigotimes_{i=1}^n (P_\la \circ \ev_{B_i}^{-1})$
for all integers $n \ge 1$ and mutually disjoint $B_1,\dots,B_n \in \cS$.
\item $P_\la \circ \ev_{B}^{-1} = \Poi(\la(B))$ for all $B \in \cS$.
\end{enumerate}
For the existence and uniqueness, see e.g.\ \cite[Proposition 2.10, Theorem 3.6]{Last_Penrose_2018}
or \cite[Lemma 12.1--12.2, Theorem 12.7]{Kallenberg_2002}.
Samples $\eta$ from $P_\la$ are called \new{Poisson PPs}.

\subsection{\Renyi divergences}

\Renyi divergences were introduced in \cite{Renyi_1961}.  The \new{\Renyi divergence} of order $\al \in [0,\infty)$ for probability measures $P,Q$ on a measurable space $(S,\cS)$ is defined \cite{VanErven_Harremoes_2014,Polyanskiy_Wu_2024} by
\begin{equation}
 \label{eq:Renyi}
 \Ren_\al(P \| Q)
 \weq
 \begin{cases}
  - \log Q(p>0), &\quad \al = 0,\\
  \frac{1}{\al-1} \log \int_S p^\al q^{1-\al} \, d\nu, &\quad \al \notin \{0,1\}, \\
  \int_S p \log \frac{p}{q} \, d\nu, &\quad \al=1,
 \end{cases}
\end{equation}
where $p = \frac{dP}{d\nu}$ and $q = \frac{dQ}{d\nu}$ are densities of $P,Q$ with respect to a sigma-finite measure\footnote{The definition does not depend on the choice of reference measure or densities, we may take $\nu = \frac12 (P+Q)$ for example \cite{VanErven_Harremoes_2014}.} $\nu$ on $S$;
and for $\al>1$ we read $p^\al q^{1-\al} = \frac{p^\al}{q^{\al-1}}$ and adopt the conventions \cite{Liese_Vajda_2006,VanErven_Harremoes_2014}
that $\frac{0}{0} = 0$ and $\frac{t}{0} = \infty$ for $t > 0$, together with $0 \log \frac{0}{t} = 0$ for $t \ge 0$ and
$t \log \frac{t}{0} = \infty$ for $t > 0$.  For $\al \in (1,\infty)$, note that
$\int_S p^\al q^{1-\al} \, d\nu = \int_{p>0, \, q>0} p^\al q^{1-\al} \, d\nu + \infty P\{q=0\}$, so that
\[
 \Ren_\al(P \| Q)
 \weq
 \begin{cases}
    \frac{1}{\al-1} \log \int_{p>0, \, q>0} p^\al q^{1-\al} \, d\nu, &\qquad P \ll Q, \\
   \infty, &\qquad P \not\ll Q.
 \end{cases}
\]
\Renyi divergences also admit several variational characterisations, see for example
\cite{Anantharam_2018,Birrell_etal_2021,VanErven_Harremoes_2014}.

\Renyi divergences of orders $\frac12, 1, 2$ are connected to other important information quantities as follows:
$\Ren_1(P,Q)$ equals the \new{Kullback--Leibler divergence} or \new{relative entropy},
$\Ren_{1/2}(P \| Q) = -2 \log ( 1 - \frac{\Hel^2(P,Q)}{2})$
where $\Hel(P,Q)$ is the Hellinger distance,
and $\Ren_2(P \| Q) = \log(1 + \chi^2(P,Q))$
where $\chi^2(P,Q)$ refers to the $\chi^2$-divergence \cite{Gibbs_Su_2002}.

\section{Tsallis divergences of sigma-finite measures}
\label{sec:Tsallis}

This section introduces a theoretical framework of Tsallis divergences of sigma-finite measures.
Section~\ref{sec:TsallisDefinition} provides a definition and a representation formula
as a Poisson--\Renyi integral, and Section~\ref{sec:TsallisProperties} summarises some basic properties.
Section~\ref{sec:TsallisAC} demonstrates how Tsallis divergences characterise absolute continuity and mutual singularity.
Section~\ref{sec:Hellinger} establishes a connection with Hellinger distances,
and Section~\ref{sec:TsallisKernel} presents a disintegration formula.

\subsection{Definition}
\label{sec:TsallisDefinition}

Tsallis divergences of probability measures were introduced in \cite{Tsallis_1998} (see also \cite{Nielsen_Nock_2012}).
The following definition generalises the notion of Tsallis divergence from probability measures to arbitrary sigma-finite measures.
Let $\la, \mu$ be sigma-finite measures on a measurable space $S$
admitting densities $f = \frac{d\la}{d\nu}$ and $g = \frac{d\mu}{d\nu}$ with respect to a sigma-finite measure $\nu$.
The \new{Tsallis divergence} of order $\alpha \in \R_+$ is defined by
\begin{equation}
 \label{eq:Tsallis}
 \Tsa_\al( \la \| \mu )
 =
 \begin{cases}
 \mu\{f = 0\}, &\quad \al = 0, \\
 \int_S \big( \frac{\al f + (1-\al)g - f^\al g^{1-\al}}{1-\al} \big) \, d\nu, &\quad \al \notin \{0,1\}, \\
 \int_S \big( f \log \frac{f}{g} + g - f \big) \, d\nu, &\quad \al = 1,
 \end{cases}
\end{equation}
where for $\alpha>1$ we read $f^\al g^{1-\al}$ as $\frac{f^\al}{g^{\al-1}}$ and adopt the 
conventions that $\frac{0}{0} = 0$ and $\frac{t}{0} = \infty$ for $t > 0$, as well as
$t \log \frac{t}{0} = \infty$ for $t>0$ and $0 \log \frac{0}{t} = 0$ for $t \ge 0$.

\begin{theorem}
\label{the:Tsallis}
$\al \mapsto \Tsa_\al(\la \| \mu)$ is a well-defined nondecreasing function from $\R_+$ into $[0,\infty]$
that is continuous on the interval $\{\al \colon \Tsa_\al(\la \| \mu) < \infty\}$, and admits
a representation
\begin{equation}
 \label{eq:TsallisPoissonNew}
 \Tsa_\al(\la \| \mu)
 \weq \int_S \Ren_\al(p_{f(x)} \| p_{g(x)}) \, \nu(dx),
\end{equation}
where $p_s$ refers to the Poisson distribution $k \mapsto e^{-s k} \frac{s^k}{k!}$ on the nonnegative integers 
with mean $s$.  Furthermore, the value of $\Tsa_\al(\la \| \mu)$ does not depend on the choice of the densities $f,g$ nor the measure $\nu$.
\end{theorem}
\begin{proof}
Section~\ref{sec:ProofTsallis}.
\end{proof}

\begin{remark}
In the special case with $\la(S)=\mu(S)=1$, we find that
$\Tsa_\al( \la \| \mu ) = \frac{1-\int_S f^\al g^{1-\al} \, d\nu}{1-\alpha}$ for $\alpha \in (0,1)$
agrees with the classical definition of the Tsallis divergence for probability measures \cite{Tsallis_1998}.
In this case $\Tsa_1(\la \| \mu) = \KL(\la \| \mu)$ equals the Kullback--Leibler divergence, and
the Tsallis divergence of order $\al \notin \{0,1\}$ is related to the \Renyi divergence by
$\Tsa_\al(\la \| \mu) = \frac{1-e^{-(1-\al) \Ren_\al(\la \| \mu)}}{1-\al}$.
In case of finite measures, the formulas on the right side of \eqref{eq:Tsallis} can be simplified by 
replacing $\int_S f \, d\nu = \la(S)$ and $\int_S g \, d\nu = \mu(S)$.
Such simplifications are not possible for general sigma-finite measures,
but Theorem~\ref{the:Tsallis} guarantees that the integrals in \eqref{eq:Tsallis}
are nevertheless well defined.
\end{remark}

\begin{remark}
$\Tsa_{1/2}(\la \| \mu) = 2 H^2(\la,\mu)$ where $H(\la,\mu)$ refers to the Hellinger distance (see Section~\ref{sec:Hellinger}), and $\Tsa_1(\la \| \mu)$ corresponds to the generalized KL divergence discussed in \cite{Miller_Federici_Weniger_Forre_2023}.
\end{remark}

\begin{remark}
When $\la,\mu$ admit strictly positive densities $f,g$ with respect to a sigma-finite measure $\nu$,
the Tsallis divergence of order $\alpha \ne 1$ can be written as $\Tsa_\al( \la \| \mu ) = \int g \, \Phi( \frac{f}{g} ) \, d\nu$ where $\Phi(t) = \frac{t^\alpha - 1 - \alpha(t-1)}{\alpha-1}$ is a convex function such that $\Phi(1)=0$.
In this sense $\Tsa_\al( \la \| \mu )$ corresponds to an instance of an f-divergence between sigma-finite measures $\la,\mu$.  Tsallis divergences restricted to probability measures may therefore be analysed using the rich theory of f-divergences \cite{Polyanskiy_Wu_2024} \cite{Sason_2018}.
\end{remark}

\subsection{Properties}
\label{sec:TsallisProperties}

Tsallis divergences share several properties in common with \Renyi divergences.
This section summarises some of the most important.

\begin{proposition}
\label{the:TsallisAC}
The Tsallis divergence for sigma-finite measures $\la \ll \mu$ with density $\phi = \frac{d\la}{d\mu}$ can be written as
\begin{equation}
 \label{eq:TsallisAC}
 \Tsa_\al( \la \| \mu )
 =
 \begin{cases}
 \mu\{ \phi = 0 \}, &\qquad \al = 0, \\
 \int_S \big( \frac{ \al \phi + 1-\al - \phi^\al }{1-\al} \big) d\mu, &\qquad \alpha \notin \{0,1\}, \\
 \int_S \left( \phi \log \phi + 1 - \phi \right) d\mu, &\qquad \alpha = 1.
 \end{cases}
\end{equation}
\end{proposition}
\begin{proof}
Theorem~\ref{the:Tsallis}
indicates that we are free to choose densities and reference measures in
\eqref{eq:Tsallis}.  The claim follows by choosing densities $f=\phi$, $g=1$ of $\la,\mu$ with respect
to reference measure $\nu=\mu$.
\end{proof}

\begin{proposition}
\label{the:TsallisSkewSymmetry}
$(1-\al) \Tsa_\al(\mu \| \la) = \al \Tsa_{1-\al}(\la \| \mu)$ for all $\alpha \in (0,1)$.
\end{proposition}
\begin{proof}
Immediate from formula \eqref{eq:Tsallis}.
\end{proof}

\begin{proposition}
\label{the:TsallisBounds}
$\frac{\al}{\beta} \frac{1-\beta}{1-\al} \Tsa_\beta(\la \| \mu) \le \Tsa_\al(\la \| \mu) \le \Tsa_\beta(\la \| \mu)$
for all $\al \le \beta$ in $(0,1)$.
\end{proposition}
\begin{proof}
The second inequality follows by the monotonicity of Tsallis divergences (Theorem~\ref{the:Tsallis}).
The first inequality follows by applying Proposition~\ref{the:TsallisSkewSymmetry} and monotonicity
to conclude that
\begin{align*}
 \frac{\al}{\beta} \frac{1-\beta}{1-\al} \Tsa_\beta(\la \| \mu)
 &\weq \frac{\al}{1-\al} \Tsa_{1-\beta}(\mu \| \la) \\
 &\wle \frac{\al}{1-\al} \Tsa_{1-\al}(\mu \| \la)
 \weq \Tsa_\al(\la \| \mu).
\end{align*}
\end{proof}

\subsection{Absolute continuity and mutual singularity}
\label{sec:TsallisAC}

The following is a characterisation of absolute continuity for sigma-finite measures in terms of Tsallis divergences.
It is similar in spirit for an analogous characterisation of probability measures using \Renyi divergences:
$P \ll Q$ iff $\Ren_0(Q \| P) = 0$ (see \cite[Theorem III.9.2]{Shiryaev_1996}, \cite[Theorem 23]{VanErven_Harremoes_2014}).

\begin{proposition}
\label{the:AbsoluteContinuityMeasure}
The following are equivalent for all sigma-finite measures:
\begin{enumerate}[(i)]
\item $\la \ll \mu$.
\item $\Tsa_0(\mu \| \la) = 0$.
\end{enumerate}
\end{proposition}
\begin{proof}
Let $\la, \mu$ be measures on a measurable space $(S, \cS)$ admitting densities $f = \frac{d\la}{d\nu}$ and $g = \frac{d\mu}{d\nu}$ with respect to a measure~$\nu$. 
The equivalence of (i) and (ii) follows
by applying Lemma~\ref{the:Density} and noting that 
$\la\{g=0\} = \Tsa_0(\mu \| \la)$ by definition \eqref{eq:Tsallis}.

\end{proof}

The following is a characterisation of mutual singularity for finite measures in terms of Tsallis divergences.
It is similar in spirit to an analogous characterisation of probability measures using \Renyi divergences: $P \perp Q$ iff $\Ren_0(P \| Q) = \infty$ iff $\Ren_0(Q \| P) = \infty$ (see \cite[Theorem III.9.3]{Shiryaev_1996}, \cite[Theorem 24]{VanErven_Harremoes_2014}).

\begin{proposition}
\label{the:SingularityFiniteMeasures}
The following are equivalent for all finite measures:
\begin{enumerate}[(i)]
\item $\la \perp \mu$.
\item $\Tsa_0(\la \| \mu) = \mu(S)$.
\item $\Tsa_0(\mu \| \la) = \la(S)$.
\end{enumerate}
\end{proposition}
\begin{proof}
Let $\la, \mu$ be measures on a measurable space $(S, \cS)$ admitting densities $f = \frac{d\la}{d\nu}$ and $g = \frac{d\mu}{d\nu}$ with respect to a measure~$\nu$. 

(i)$\iff$(ii). By \eqref{eq:Tsallis}, $\Tsa_0(\la \| \mu) = \mu\{f = 0\}$. Hence $\mu(S) - \Tsa_0(\la \| \mu) = \mu\{f>0\}$.
The claim follows because $\la \perp \mu$ is equivalent to $\mu\{f>0\} = 0$ (Lemma~\ref{the:Density}).

(i)$\iff$(iii). Analogously, $\la(S) - \Tsa_0(\mu \| \la) = \la(S) - \la\{g=0\} = \la\{g>0\}$.
The claim follows because $\la \perp \mu$ is equivalent to $\la\{g>0\} = 0$ (Lemma~\ref{the:Density}).

\end{proof}

\subsection{Hellinger distances of sigma-finite measures}
\label{sec:Hellinger}

The \new{Hellinger distance} between sigma-finite measures $\la$ and $\mu$ is defined by
\begin{equation}
 \label{eq:Hellinger}
 H(\la,\mu) \weq \left( \frac12 \int_S \left(\sqrt{f}-\sqrt{g} \right)^2 d\nu \right)^{1/2},
\end{equation}
where $f=\frac{d\la}{d\nu}$ and $g=\frac{d\mu}{d\nu}$ are densities with respect to a sigma-finite measure~$\nu$.  Hellinger distances take values in $[0,1]$ for probability measures, and in 
$[0,\infty]$ for general sigma-finite measures.  By writing $(\sqrt{f}-\sqrt{g})^2 = f + g - 2 f^{1/2} g^{1/2}$,  we see by comparing \eqref{eq:Tsallis} and \eqref{eq:Hellinger} that
\begin{equation}
\label{eq:HellingerTsallis}
 \Tsa_{1/2}(\la \| \mu)
 \weq 2 H^2(\la,\mu).
\end{equation}
In particular, we see see that Tsallis divergences of order $\frac12$ are symmetric according to $\Tsa_{1/2}(\la \| \mu) = \Tsa_{1/2}(\mu \| \la)$.  When $\la,\mu$ are probability measures, we note that 
$\Tsa_{1/2}(\la \| \mu) = 2(1 - \int_S f^{1/2} g^{1/2} \, d\nu) = 2(1 - \exp( -\frac12 \Ren_{1/2}(\la \| \mu))$.
Hence for probability measures $\la,\mu$,
\begin{equation}
\label{eq:HellingerRenyi}
 H^2(\la, \mu)
 \weq 1 - \exp\left(-\frac12 \Ren_{1/2}(\la \| \mu)\right).
\end{equation}

\begin{proposition}
\label{the:Hellinger}
The right side of \eqref{eq:Hellinger} does not depend on the choice of the densities $f,g$ nor the reference measure $\nu$.
Furthermore, if $\la \ll \mu$, then the Hellinger distance can also be written as 
\[
 H(\la,\mu)
 \weq \left( \frac12 \int_S \left(\sqrt{\phi} - 1 \right)^2  \, d\mu \right)^{1/2},
\]
where $\phi \colon S \to \R_+$ is a density of $\la$ with respect to $\mu$.
\end{proposition}
\begin{proof}
The first claim follows by applying Theorem~\ref{the:Tsallis} with \eqref{eq:HellingerTsallis}.
The second claim follows by applying \eqref{eq:Hellinger} with
$f=\phi$, $g=1$, and $\nu=\mu$.
\end{proof}

\begin{proposition}
\label{the:HellingerTriangle}
$H(\la,\xi) \le H(\la,\mu) + H(\mu,\xi)$ for all sigma-finite measures.
\end{proposition}
\begin{proof}
Let $f,g,h$ be densities of $\la,\mu,\xi$ with respect to the sigma-finite measure $\nu = \la+\mu+\xi$.
We note that
\begin{align*}
 H(\la,\xi)
 &\weq \frac{1}{\sqrt{2}} \norm{ \sqrt{f} - \sqrt{h} }_{L^2(\nu)} \\
 &\weq \frac{1}{\sqrt{2}} \norm{ (\sqrt{f} - \sqrt{g}) + (\sqrt{g} - \sqrt{h}) }_{L^2(\nu)}.
\end{align*}
Therefore, the claim follows by applying Minkowski's inequality
$\norm{u+v}_{L^2(\nu)}
\le \norm{u}_{L^2(\nu)} + \norm{v}_{L^2(\nu)}$
which is true for all measurable functions $u,v$, regardless of whether $\norm{u}_{L^2(\nu)}, \norm{v}_{L^2(\nu)}$ are finite or not.
\end{proof}

\subsection{Disintegration of Tsallis divergences}
\label{sec:TsallisKernel}

The disintegration of a measure $\Lambda$ on a product space $(S_1 \times S_2, \cS_1 \otimes \cS_2)$ refers to
a representation $\Lambda = \la \otimes K$ where
$\la$ is a measure on $(S_1,\cS_1)$ corresponding to the first marginal of $\Lambda$,  and
$K$ is a kernel from $(S_1,\cS_1)$ into $(S_2,\cS_2)$.  Equivalently,
\[
 \Lambda(C)
 \weq \int_{S_1} \int_{S_2} 1_C(x,y) \, K_x(dy) \, \la(dx),
 \quad C \in \cS_1 \otimes \cS_2.
\]
Informally, we write
\[
 \Lambda(dx,dy) \weq \la(dx) K_x(dy).
\]

If $\Lambda$ disintegrates according to $\Lambda = \la \otimes K$, then
it also disintegrates according to $\Lambda = \tilde \la \otimes \tilde K$ where
$\tilde\la = 2 \la$ and $\tilde K = \frac12 K$.
To rule out such unidentifiability issues, in applications it is natural
to require $K$ to be probability kernel, that is, a kernel such that $K_x(S_2) = 1$ for all $x \in S_1$.
The following result characterises Tsallis divergences of disintegrated measures
that helps to compute information divergences of compound Poisson processes and marked Poisson PPs
(see Sections~\ref{sec:MarkedPPP}--\ref{sec:CompoundPoissonProcesses}).
See \cite[Theorem 2.13, Equation 7.71]{Polyanskiy_Wu_2024} for similar results
concerning \Renyi divergences.

\begin{theorem}
\label{the:TsallisKernel}
Let $\la,\mu$ be sigma-finite measures on a measurable space $(S_1, \cS_1)$, and
let $K,L$ be probability kernels from $(S_1, \cS_1)$ into a measurable space $(S_2,\cS_2)$.
Assume that there exist measurable functions $k, \ell \colon S_1 \times S_2 \to \R_+$
and a kernel $M$ from $(S_1, \cS_1)$ into $(S_2,\cS_2)$ such that
\begin{equation}
 \label{eq:TsallisKernelMeasurable}
 \begin{aligned}
 K_t(dx) &= k_t(x) M_t(dx), \\
 L_t(dx) &= \ell_t(x) M_t(dx), \\
 \end{aligned}
\end{equation}
for all $t \in S_1$.
Then the Tsallis divergence of order $\alpha \in \R_+$ equals
\begin{equation}
 \label{eq:TsallisKernel}
 \begin{aligned}
 &\Tsa_\al( \la \otimes K \| \mu \otimes L ) \\
 &\weq
 \begin{cases}
 T_0(\la \| \mu) + \int_{f \ne 0} \Tsa_0(K_t \| L_t) \, \mu(dt),
 &\quad \al = 0, \\
 \Tsa_\al(\la \| \mu) + \int_{S_1} \Tsa_\al(K_t \| L_t) \, f_t^\al g_t^{1-\al} \, \nu(dt),
 &\quad \al \notin \{0,1\}, \\
 \Tsa_1(\la \| \mu) + \int_{S_1} \Tsa_1( K_t \| L_t ) \, \la(dt),
 &\quad \al = 1,
 \end{cases}
 \end{aligned}
\end{equation}
where $f_t = \frac{d\la}{d\nu}(t)$ and $g_t = \frac{d\mu}{d\nu}(t)$ are densities of $\la$ and $\mu$ with respect to a sigma-finite measure~$\nu$.
\end{theorem}
\begin{proof}
Section~\ref{sec:TsallisKernelProof}.
\end{proof}

A sufficient condition for \eqref{eq:TsallisKernelMeasurable} is to assume that 
$(S_2, \cS_2)$ is separable in the sense that $\cS_2$ is generated by a countable set family.
In this case 
there exist \cite[Theorem 58]{Dellacherie_Meyer_1982} measurable functions $k,\ell \colon S_1 \times S_2 \to \R_+$
such that \eqref{eq:TsallisKernelMeasurable} holds for the probability kernel
$M = \frac12(K+L)$.
It might be that \eqref{eq:TsallisKernelMeasurable} is not needed for Theorem~\ref{the:TsallisKernel}.  
Proving this might require a measurable selection theorem that is different from the usual ones for which it is assumed that
at least one of the sigma-algebras is generated by a regular topology
\cite{Leskela_2010,Leskela_Vihola_2017}.

\section{Likelihood ratios of Poisson PPs}
\label{sec:LR}

This section presents a general likelihood ratio formula for Poisson PPs.
Section~\ref{sec:PoissonDensityFinite} first focuses on the case with finite intensity measures,
and Section~\ref{sec:PoissonDensitySigmafinite} then provides a general formula.

\subsection{Finite Poisson PPs}
\label{sec:PoissonDensityFinite}


Poisson PPs with finite intensity measures are almost surely finite.
Likelihood ratio formulas for Poisson PP distributions with finite intensity measures are classical (e.g.\ \cite{Liptser_Shiryaev_1978_II,Karr_1991,Reiss_1993,Birge_2007}), although they are usually restricted to Polish spaces.  The proof of the following general result is included in Section~\ref{sec:PoissonDensityFiniteProof} for completeness.

\begin{theorem}
\label{the:PoissonDensityFinite}
Any Poisson PP distributions with finite intensity measures $\la \ll \mu$
satisfy $P_\la \ll P_\mu$, 
and a likelihood ratio is given by
\begin{equation}
 \label{eq:PoissonDensityFinite}
 \frac{dP_\la}{dP_\mu}(\eta)
 \weq 1_{M_{\la,\mu}}(\eta) \, \exp\bigg( \int_S ( 1 - \phi) \, d\mu + \int_S \log \phi \, d\eta \bigg),
\end{equation}
where
$\phi = \frac{d\la}{d\mu}$ is a density of $\la$ with respect to $\mu$, and
\begin{equation}
 \label{eq:PoissonDensitySetFinite}
 M_{\la,\mu}
 \weq \Big\{\eta \in N(S) \colon \eta\{\phi=0\} = 0 \Big\}.
\end{equation}
\end{theorem}

\begin{remark}
\label{rem:PoissonDensityFinite}
For every finite point pattern $\eta \in M_{\la,\mu}$, the integral $\int_S \log \phi \, d\eta$ in \eqref{eq:PoissonDensityFinite} is a well-defined real number because $\log \phi \in \R$ outside the set $\{\phi=0\}$ of $\eta$-measure zero.
Also recall that every point pattern generated by a Poisson PP distribution with a finite intensity measure is finite almost surely.  Therefore, the right side in \eqref{eq:PoissonDensityFinite} is well defined for $P_\mu$-almost every $\eta$.
\end{remark}

\begin{remark}
\label{rem:PoissonDensitySetFinite}
The set $M_{\la,\mu}$ in \eqref{eq:PoissonDensitySetFinite} indicates the set of point patterns that contain no points
in the region where $\phi = \frac{d\la}{d\mu}$ vanishes.  If we assume that $\la$ and $\mu$ are mutually absolutely continuous, then we may omit $1_{M_{\la,\mu}}(\eta)$ from \eqref{eq:PoissonDensitySetFinite}.
To see why, observe that 
$\la\{\phi=0\} = \int_{\{\phi = 0\}} \phi \, d\mu = 0$ 
together with $\mu \ll \la$ implies that
$\mu\{\phi=0\} = 0$. Therefore, by Markov's inequality,
\begin{align*}
 P_\mu(M_{\la,\mu}^c)
 &\weq P_\mu\{\eta\{\phi=0\} \ge 1\} \\
 &\wle E_\mu \eta\{\phi=0\}
 \weq \mu\{\phi=0\}
 \weq 0.
\end{align*}
\end{remark}

\subsection{Sigma-finite PPs}
\label{sec:PoissonDensitySigmafinite}

The simple density formula of Theorem~\ref{the:PoissonDensityFinite}
is not in general valid for Poisson PPs with infinite intensity measures, because 
the integral $\int_S \log \phi \, d\eta$ might not converge for infinite point patterns $\eta$.
In the general setting, we need to work with carefully compensated integrals.
A key observation is that a compensated Poisson integral
$
 \int_{ \{\abs{\log \phi} \le 1\} } \log\phi \, d(\eta-\mu)
$
and the ordinary Poisson integral $\int_{ \{\abs{\log \phi} > 1\} } \log \phi \, d\eta$
of the logarithm of $\phi = \frac{d\la}{d\mu}$ 
%
%
converge for $P_\mu$-almost every $\eta$
whenever $P_\la \ll P_\mu$.
See Appendix~\ref{sec:CompensatedPoissonIntegral} for the definition of the compensated integral
and details.  The following theorem confirms that a density of $P_\la$ with respect to $P_\mu$ can be written
using these integrals.



\begin{theorem}
\label{the:PoissonDensity}
Any Poisson PP distributions with sigma-finite intensity measures such that $\la \ll \mu$ and $H(\la,\mu) < \infty$
satisfy $P_\la \ll P_\mu$, and
a likelihood ratio is given by
\begin{equation}
 \label{eq:PoissonDensity}
 \frac{dP_\la}{dP_\mu}(\eta)
 \weq
 1_{M_{\la,\mu}}(\eta) \exp(\ell_{\la,\mu}(\eta)),
\end{equation}
where
\begin{equation}
 \label{eq:PoissonDensityExponent}
 \begin{aligned}
 \ell_{\la,\mu}(\eta)
 & \ = \nhquad \intlim_{\abs{\log \phi} \le 1} \log\phi \, d(\eta-\mu)
 \, + \nhquad \intlim_{\abs{\log \phi} > 1} \log \phi \, d\eta \\
 & \quad + \nhquad \intlim_{\abs{\log \phi} \le 1}  (\log \phi + 1-\phi) \, d\mu
 + \nhquad \intlim_{\abs{\log \phi} > 1} (1-\phi) \, d\mu
 \end{aligned}
\end{equation}
and
\begin{equation}
 \label{eq:PoissonDensitySet}
 M_{\la,\mu}
 \weq \{\eta \in N(S) \colon \eta\{\phi=0\} = 0\}
\end{equation}
are defined in terms of a density $\phi = \frac{d\la}{d\mu}$.
\end{theorem}
\begin{proof}
Section~\ref{sec:PoissonDensitySigmafiniteProof}.
\end{proof}

\begin{remark}
When $\la$ and $\mu$ are mutually absolutely continuous, the factor $1_{M_{\la,\mu}}(\eta)$ may be omitted from \eqref{eq:PoissonDensity}, as explained in Remark~\ref{rem:PoissonDensitySetFinite}.
\end{remark}

\section{Information divergences of Poisson PPs}
\label{sec:Divergences}

In principle, most information divergences for Poisson PP distributions can be computed using the likelihood ratio $\frac{dP_\la}{dP_\mu}$.  Unfortunately, the general likelihood ratio formula in Theorem~\ref{the:PoissonDensity} involves a rather complicated stochastic integral that renders it difficult to obtain explicit analytical expressions.
However, the fact that the laws of Poisson PPs are infinitely divisible suggests that
simple formulas should be available for information divergences that are additive with respect to
product measures.  It is well known that R{\'e}nyi divergences, including the Kullback--Leibler divergence, enjoy this tensorisation property \cite{VanErven_Harremoes_2014,Polyanskiy_Wu_2024}.
Indeed, it was recently confirmed that linear combinations of R{\'e}nyi divergences are the only divergences satisfying the tensorisation property and the data processing inequality \cite[Theorem 2]{Mu_Pomatto_Strack_Tamuz_2021}.

This section demonstrates how R{\'e}nyi divergences and related quantities of general Poisson PP distributions can be computed from their associated intensity measures as generalised Tsallis divergences introduced in Section~\ref{sec:Tsallis}.  The section is outlined as follows.
Section~\ref{sec:PPPRenyi} summarises formulas \Renyi divergences, Kullback--Leibler divergences, and Hellinger distances.
Section~\ref{sec:PPPAC} characterises the absolute continuity of Poisson PP distributions using Tsallis divergences of their intensity measures.  
Section~\ref{sec:PPPDomination} characterises pairs of Poisson PPs whose laws admit a common dominating measure corresponding to a Poisson PP.

\subsection{Divergences and distances}
\label{sec:PPPRenyi}

In what follows, $P_\la$ and $P_\mu$ are Poisson PP distributions with sigma-finite intensity measures $\la$ and $\mu$ on a measurable space $S$.

\begin{theorem}
\label{the:PPPRenyi}
The \Renyi divergence of order $\al \in (0,\infty)$ for Poisson PP distributions $P_\la$ and $P_\mu$
is given by the Tsallis divergence of their intensity measures according to
\begin{equation}
 \label{eq:PoissonRenyi}
 \Ren_\al( P_\la \| P_\mu )
 \weq \Tsa_\alpha(\la \| \mu).
\end{equation}
If $\Tsa_\alpha(\la \| \mu) < \infty$ for some $\al > 0$, then 
\eqref{eq:PoissonRenyi} also holds for $\al=0$.
\end{theorem}
\begin{proof}
Section~\ref{sec:PoissonRenyiProof}.
\end{proof}

\begin{theorem}
\label{the:PoissonKL}
The Kullback--Leibler divergence for Poisson PP distributions $P_\la$ and $P_\mu$
is given by
\begin{equation}
 \label{eq:PoissonKL}
 \KL( P_\la \| P_\mu )
 \weq \int_S \left( f \log \frac{f}{g} + g - f \right) d\nu,
\end{equation}
where $f = \frac{d\la}{d\nu}$ and $g = \frac{d\mu}{d\nu}$ are densities with respect to
a measure $\nu$.
\end{theorem}
\begin{proof}
Immediate corollary of Theorem~\ref{the:PPPRenyi}.
\end{proof}

As another corollary of Theorem~\ref{the:PPPRenyi}, we obtain a simple formula
for the Hellinger distance between Poisson PP distributions.
This result was proved in \cite{Takahashi_1990}
for
locally finite intensity measures on locally compact Polish spaces.

\begin{theorem}
\label{the:PoissonHellinger}
The Hellinger distance between Poisson PPs $P_\la$ and $P_\mu$
is given by $H^2(P_\la, P_\mu) = 1 - e^{-H^2(\la,\mu)}$.
\end{theorem}
\begin{proof}
Theorem~\ref{the:PPPRenyi} combined with formula \eqref{eq:HellingerTsallis}
implies that $\Ren_{1/2}( P_\la \| P_\mu ) = \Tsa_{1/2}(\la \| \mu) = 2 H^2(\la,\mu)$.
For probability measures $P_\la$ and $P_\mu$, we find by applying \eqref{eq:Renyi} that
$H^2(P_\la, P_\mu) = 1 - e^{-\frac12 \Ren_{1/2}(P_\la, P_\mu)}$.
By combining these findings, we conclude that 
$H^2(P_\la, P_\mu) = 1 - e^{-H^2(\la,\mu)}$.
\end{proof}

\subsection{Absolute continuity}
\label{sec:PPPAC}

\begin{theorem}
\label{the:AbsoluteContinuityPoissonTsallis}
The following are equivalent for any Poisson PPs with sigma-finite intensity measures:
\begin{enumerate}[(i)]
\item $P_\la \ll P_\mu$.
\item $\Tsa_0(\mu \| \la) = 0$ and $\Tsa_\al(\mu \| \la) < \infty$ for all $\al \in \R_+$.
\item $\Tsa_0(\mu \| \la) = 0$ and $\Tsa_\al(\mu \| \la) < \infty$ for some $0 < \al < \infty$.
\end{enumerate}
\end{theorem}
\begin{proof}
(i)$\implies$(ii).
Assume that $P_\la \ll P_\mu$.  Fix a measurable set $A \subset S$ such that $\mu(A) = 0$.
Let $C = \{\eta \in N(S) \colon \eta(A) > 0\}$.  Note that $\eta(A)$ is Poisson-distributed with mean $\la(A)$ (resp.\ $\mu(A)$) when $\eta$ is sampled from $P_\la$ (resp.\ $P_\mu$). Therefore $P_\mu(C) = 1-e^{-\mu(A)} = 0$.
Because $P_\la \ll P_\mu$, it follows that
$0 = P_\la( C ) = 1 - e^{-\la(A)}$, and we conclude that $\la(A) = 0$.  Hence $\la \ll \mu$, and
Proposition~\ref{the:AbsoluteContinuityMeasure} implies that $\Tsa_0(\mu \| \la)=0$.
Furthermore, because $P_\la$ and $P_\mu$ are not mutually singular,
we know \cite[Theorem 24]{VanErven_Harremoes_2014} that $\Ren_\al(P_\mu \| P_\la) < \infty$ for all
$\al \in \R_+$. By Theorem~\ref{the:PPPRenyi}, it follows that $\Tsa_\al(\mu \| \la) < \infty$ for all
$\al \in \R_+$.

(ii)$\implies$(iii).
Immediate.

(iii)$\implies$(i).
Theorem~\ref{the:PPPRenyi} implies that $\Ren_0(P_\mu \| P_\la) = \Tsa_0(\mu \| \la) = 0$.
By formula~\eqref{eq:Renyi}, we see that $-\log P_\la\{G>0\} = 0$ where $G = \frac{dP_{\mu}}{dm}$ is a
density of $P_\mu$ with respect to an arbitrary measure $m$ such that $P_\la,P_\mu \ll m$.
Hence $P_\la\{G=0\} = 0$, and Lemma~\ref{the:Density} confirms that $P_\la \ll P_\mu$.
\end{proof}

As a corollary of Theorem~\ref{the:AbsoluteContinuityPoissonTsallis},  we obtain a simple proof
of the following result extending \cite{Takahashi_1990} to general nontopological spaces.

\begin{theorem}
\label{the:AbsoluteContinuityPoissonHellinger}
The following are equivalent for Poisson PPs with sigma-finite intensity measures:
\begin{enumerate}[(i)]
\item $P_\la \ll P_\mu$
\item $\la \ll \mu$ and $H(\la,\mu) < \infty$.
\end{enumerate}
\end{theorem}
\begin{proof}
Assume that $\la \ll \mu$ and $H(\la,\mu) < \infty$.
Then $\Tsa_0(\mu \| \la) = 0$ by Proposition~\ref{the:AbsoluteContinuityMeasure}.
Formula \eqref{eq:HellingerTsallis} implies that $\Tsa_{1/2}(\mu \| \la) = 2 H^2(\mu,\la) = 2 H^2(\la,\mu)$ is finite.
Theorem~\ref{the:AbsoluteContinuityPoissonTsallis} now implies that $P_\la \ll P_\mu$.

Assume that $P_\la \ll P_\mu$.
Theorem~\ref{the:AbsoluteContinuityPoissonTsallis} then implies that $\Tsa_0(\mu \| \la) = 0$
and $\Tsa_{1/2}(\mu \| \la) < \infty$. Then Proposition~\ref{the:AbsoluteContinuityMeasure} implies that $\la \ll \mu$,
and formula \eqref{eq:HellingerTsallis} implies that $H(\la,\mu) < \infty$.
\end{proof}

Kakutani's famous dichotomy \cite{Kakutani_1948}
states that infinite products of probability measures $\prod_i P_i$ and $\prod_i Q_i$, such that $P_i$ and $Q_i$ are mutually absolutely continuous for all $i$, are either mutually absolutely continuous or mutually singular---there is no middle ground.
Earlier results in this and the previous section yield
a simple proof of an analogue of Kakutani's dichotomy for Poisson PPs.
This result is well known for Polish spaces
\cite{Liese_1975,Karr_1991}, and has also been presented in
\cite{Brown_1971} in terms of a more complicated criterion for
intensity measures that is equivalent to $H(\la,\mu) < \infty$.

\begin{theorem}
\label{the:PoissonKakutani}
Let $\la$ and $\mu$ be mutually absolutely continuous. Then $P_\la$ and $P_\mu$ are either mutually absolutely continuous, or mutually singular, according to $H(\la,\mu)$ being finite or infinite.
\end{theorem}
\begin{proof}
Assume that $H(\la,\mu) < \infty$.  Theorem~\ref{the:AbsoluteContinuityPoissonHellinger} then shows that $P_\la \ll P_\mu$ and $P_\mu \ll P_\la$,
so that $P_\la$ and $P_\mu$ are mutually absolutely continuous.

Assume next that $H(\la,\mu) = \infty$.  Theorem~\ref{the:PoissonHellinger} then implies that $H(P_\la, P_\mu) = 1$.
In light of \eqref{eq:HellingerRenyi} we see that
$H^2(P_\la, P_\mu) = 1 - \exp\left(-\frac12 \Ren_{1/2}(P_\la \| P_\mu)\right)$,
from which we conclude that $\Ren_{1/2}(P_\la \| P_\mu) = \infty$.
It follows \cite[Theorem 24]{VanErven_Harremoes_2014} that
$P_\la$ and $P_\mu$ are mutually singular.
\end{proof}

\subsection{Existence of a dominating Poisson PP}
\label{sec:PPPDomination}

For any pair of Poisson PP distributions $P_\la,P_\mu$, there always exists a probability measure $Q$ on $(N(S), \cN(S))$ such that $P_\la \ll Q$ and $P_\mu \ll Q$.  For example, we may choose $Q = \frac12( P_\la + P_\mu)$.
For practical purposes (e.g.\ Monte Carlo simulation), it would be helpful to find a
Poisson PP distribution, such as $Q = P_{\la+\mu}$, that would also serve as a
common dominating probability measure.
This is always possible for finite intensity measures but fails in general (Remark~\ref{rem:IntensitySum}).
Remarkably, even when a dominating Poisson PP exists,
$\la+\mu$ might not be a feasible choice for its intensity.

\begin{theorem}
\label{the:DominatingPPP}
Given Poisson PP distributions $P_\la,P_\mu$ with sigma-finite intensity measures $\la,\mu$,
there exists a Poisson PP distribution $P_\xi$ such that $P_\la, P_\mu \ll P_\xi$ if and only if
$H(\la,\mu) < \infty$.
\end{theorem}
\begin{proof}
Assume that there exists a Poisson PP distribution $P_\xi$ with a sigma-finite intensity measure $\xi$
such that $P_\la, P_\mu \ll P_\xi$.  Theorem~\ref{the:AbsoluteContinuityPoissonHellinger} implies that
$H(\la,\xi) < \infty$ and $H(\xi,\mu) < \infty$.  The triangle inequality (Proposition~\ref{the:HellingerTriangle})
implies that $H(\la,\mu) \le H(\la,\xi) + H(\xi,\mu)$.  Hence $H(\la,\mu) < \infty$.

Assume that $H(\la,\mu) < \infty$.  Let $f,g$ be densities of $\la,\mu$ with respect to the sigma-finite measure
$\nu = \la+\mu$.  Define a measure $\xi(dx) = h(x) \nu(dx)$ where
$h = \frac14( \sqrt{f} + \sqrt{g} )^2$.
Observe that $f \le 4h$ and $g \le 4h$ pointwise, and therefore we see that
$\la \ll \xi$ and $\mu \ll \xi$.
Furthermore, because
$\sqrt{h} = \frac12 ( \sqrt{f} + \sqrt{g} )$, we see that
$\sqrt{h}-\sqrt{f} = \frac12( \sqrt{g}-\sqrt{f})$
and
$\sqrt{h}-\sqrt{g} = \frac12( \sqrt{f}-\sqrt{g})$.
Then by formula~\eqref{eq:Hellinger},
\begin{align*}
 H^2(\la,\xi)
 \weq \frac12 \int (\sqrt{h}-\sqrt{f})^2 \, d\nu
 &\weq \frac18 \int (\sqrt{f}-\sqrt{g})^2 \, d\nu \\
 &\weq \frac14 H^2(\la,\mu).
\end{align*}
Hence $H(\la,\xi) = \frac{1}{2} H(\la,\mu)$.
By symmetry, $H(\mu,\xi) = \frac{1}{2} H(\la,\mu)$.
We conclude that $H(\la,\xi)$ and $H(\mu,\xi)$ are finite.
Theorem~\ref{the:AbsoluteContinuityPoissonHellinger} now confirms that
$P_\la \ll P_\xi$ and $P_\mu \ll P_\xi$.
\end{proof}

\begin{remark}
\label{rem:IntensitySum}
$P_\la , P_\mu \ll P_{\la+\mu}$ if and only if $H(\la,\la+\mu) + H(\mu,\la+\mu) < \infty$.
The latter requirement is stronger than $H(\la,\mu) < \infty$.  
To see why, observe that 
(see Proposition~\ref{the:Hellinger}) $H^2(\la,2\la)
= \frac12 \int_S (\sqrt{2}-\sqrt{1})^2 \, d\la
= \frac12 (\sqrt{2}-1)^2 \la(S)$.
Therefore
$H(\la,\mu)=0$ and $H(\la,\la+\mu)=\infty$ 
whenever $\la$ is an infinite measure and $\mu=\la$.
Let us also note that $P_\la , P_\mu \ll P_{\la+\mu}$ is always true for finite intensity measures $\la,\mu$
because Hellinger distances between finite measures are finite; as is seen from \eqref{eq:Hellinger}
combined with the inequality $(\sqrt{f}-\sqrt{g})^2 \le f+g$.
\end{remark}

\section{Applications}
\label{sec:Applications}

This section lists various applications of the general formulas derived in the previous sections.
Section~\ref{sec:PoissonProcesses} discusses Poisson processes,
Section~\ref{sec:ChernoffInformation} Chernoff information of Poisson vectors,
Section~\ref{sec:MarkedPPP} marked Poisson PPs,
and Section~\ref{sec:CompoundPoissonProcesses} concludes with compound Poisson processes.

\subsection{Poisson processes}
\label{sec:PoissonProcesses}

The \new{counting process} of a point pattern $\eta \in N(\R_+)$ is a function $X \colon \R_+ \to \Z_+$ defined by
\begin{equation}
 \label{eq:CountingProcess}
 X_t
 \weq \int_{\R_+} 1(s \le t) \, \eta(ds)
 \weq \eta([0,t]).
\end{equation}
Denote by $\cou \colon \eta \mapsto X$ the map induced by \eqref{eq:CountingProcess}.
A \new{Poisson process} with intensity measure $\la$ is the counting process
$X = \cou(\eta)$ of a point pattern $\eta$ sampled
from a Poisson PP distribution $P_\la$ with a sigma-finite intensity measure $\la$ on $\R_+$.
The law of the Poisson process is the pushforward measure $\law(X) = P_\la \circ \cou^{-1}$.
When $\la$ admits a density $f$ with respect to the Lebesgue measure on $\R_+$,
the process $X = (X_t)_{t \in \R_+}$ is called 
an inhomogeneous Poisson process with intensity function~$f$.

\begin{theorem}
\label{the:PoissonProcessRenyi}
For Poisson processes
$X = (X_t)_{t \in \R_+}$ and $Y = (Y_t)_{t \in  \R_+}$
with intensity functions $f$ and $g$, the Kullback--Leibler divergence is given by
\begin{align*}
 \KL( \law(X) \| \law(Y) )
 \weq \int_0^\infty \left( f_t \log \frac{f_t}{g_t} + g_t - f_t \right) dt,
\end{align*}
the \Renyi divergence or order $\al \ne 1$ is given by
\begin{align*}
 \Ren_\al( \law(X) \| \law(Y) )
 \weq \int_0^\infty \Big( \frac{\al f_t + (1-\al) g_t - f_t^\al g_t^{1-\al}}{1-\al} \Big) dt,
\end{align*}
and the Hellinger distance is given by
\[
 \Hel(\law(X), \law(Y))
 \weq \sqrt{1 - e^{-\frac12 \Tsa_{1/2}(\la \| \mu)}},
\]
where
$
 \Tsa_{1/2}(\la \| \mu)
 = \int_0^\infty \left( \sqrt{f_t}-\sqrt{g_t} \right)^2 dt.
$
\end{theorem}
\begin{proof}
Denote by $F(\R_+, \Z_+)$ set of nondecreasing functions $X \colon \R_+ \to \Z_+$ that are right-continuous with left limits (c\`adl\`ag), equipped with the sigma-algebra generated by the evaluation maps $X \mapsto X_t$.
It follows from \eqref{eq:CountingProcess} that the map $\cou \colon N(\R_+) \to F(\R_+)$
is a measurable bijection with a measurable inverse.
Therefore (Lemma~\ref{the:RenyiBijection}), $\Ren_\al( \law(X) \| \law(Y) ) = \Ren_\al( P_\la \| P_\mu )$ with $\la(dt) = f(t) dt$ and $\mu(dt) = g(t) dt$.
Theorem~\ref{the:PPPRenyi} implies that $\Ren_\al( P_\la \| P_\mu ) = \Tsa_\al( \la \| \mu )$,
and the first two claims follow by \eqref{eq:Tsallis}.

Next, we note by \eqref{eq:HellingerRenyi} that  
$H^2(\law(X), \law(Y)) = 1 - \exp\left(-\frac12 \Ren_{1/2}(\law(X) \| \law(Y))\right)$,
so that the last claim follows from $\Ren_{1/2}(\law(X) \| \law(Y)) = \Ren_{1/2}(P_\la \| P_\mu) = \Tsa_{1/2}(\la \| \mu)$.

\end{proof}

\subsection{Chernoff information for Poisson vectors}
\label{sec:ChernoffInformation}

In classical binary hypothesis testing, the task is to estimate a parameter $\theta \in \{0,1\}$
from $n$ independent samples from a probability distribution $F_\theta$ on a measurable space $S$.
The error rate (i.e.\ Bayes risk) of a decision rule $\hat \theta_n \colon S^n \to \{0,1\}$, averaged with respect to prior probabilities $\pi_0,\pi_1 > 0$ is given by $R_\pi(\hat \theta_n) = \pi_0 F_0^{\otimes n}(\hat\theta_n=1) + \pi_1 F_1^{\otimes n}(\hat\theta_n=0)$.  Chernoff's famous theorem \cite{Chernoff_1952}
states that for $n \gg 1$, the minimum error rate scales as
\[
 \inf_{\hat\theta_n} R_\pi(\hat \theta_n) \weq
 e^{-(1+o(1)) C(F_0 \| F_1) n},
\]
where
\[
 C(F_0 \| F_1)
 \weq \sup_{\al \in (0,1)} (1-\al) \Ren_\al( F_0 \| F_1 ).
\]
is called  the \new{Chernoff information} of the test \cite{Nielsen_2013}.

In a network community detection problem related to a stochastic block model, Abbe and Sandon \cite{Abbe_Sandon_2015_Community}
reduced a key estimation task into a binary hypothesis test for the law of a random vector
with independent Poisson-distributed components, either having mean vector $H_0: (\la_1,\dots,\la_K)$
or $H_1: (\mu_1,\dots,\mu_K)$.
Equivalently, the law of this random vector can be seen as a Poisson PP distribution on the finite set $\{1,\dots,K\}$
with intensity measure admitting density
$k \mapsto \la_k$
or 
$k \mapsto \mu_k$
with respect to the counting measure.
Hence we are looking at a hypothesis test between Poisson PP distributions $F_0 = P_\la$ and $F_1 = P_\mu$.  With the help of Theorem~\ref{the:PPPRenyi} and formula~\eqref{eq:Tsallis} we see that
the Chernoff information of the test is given by
\[
 C(P_\la \| P_\mu)
 \weq \sup_{\al \in (0,1)} \sum_{k=1}^K \big( \al \la_k + (1-\al)\mu_k - \la_k^\al \mu_k^{1-\al} \big).
\]
This is what is called Chernoff--Hellinger divergence in \cite{Abbe_Sandon_2015_Community,Racz_Bubeck_2017,Zhang_Tan_2023}.

\subsection{Marked Poisson PPs}
\label{sec:MarkedPPP}


A \new{marking} of a (random or nonrandom) set of points with locations $x_1,x_2,\dots$ in $S_1$ associates each point $x_i$ a random variable $y_i$ in $S_2$, called a mark, so that the conditional distribution of $y_i$ given $x_i=x$ is determined by $x$,
and the marks are conditionally independent given the locations.
The statistics of the marking mechanism is parameterised by the 
collection of conditional distributions $K \colon x \mapsto K_x = \law(y_i \cond x_i = x)$
that constitutes a probability kernel from $S_1$ into $S_2$.
Such a marking mechanism can also be defined for random point patterns
that admit a proper enumeration.  This is the case for point patterns sampled from
a Poisson PP distribution  \cite[Corollary 3.7]{Last_Penrose_2018}.

Let $P_\la$ be a Poisson PP distribution with a sigma-finite intensity measure $\la$ on a measurable space $(S_1,\cS_1)$,
and let $K$ be a probability kernel from $(S_1,\cS_1)$ into $(S_2,\cS_2)$.
A \new{marked Poisson PP} with intensity measure $\la$ and mark kernel $K$ is then 
defined by sampling a point pattern $\xi$ on $S_1 \times S_2$ from a Poisson PP distribution $P_{\la \otimes K}$ with intensity measure $\la \otimes K$ defined by 
\[
 (\la \otimes K)(C)
 \weq \int_{S_1} \int_{S_2} 1_C(x,y) \, K_x(dy) \, \la(dx),
 \qquad C \in \cS_1 \otimes \cS_2.
\]
Then the Poisson PP distribution $P_\la$ equals the law of the point pattern $A \mapsto \xi(A \times S_2)$ corresponding to the locations of the points in $S_1$, and the probability kernel $K$ 
yields the conditional distributions of the marks \cite[Section 5.2]{Last_Penrose_2018}.
Hence $\xi$ is a marked Poisson PP with intensity measure $\la$ and marking kernel $K$.

\begin{theorem}
\label{the:MarkedPPPRenyi}
If $\xi, \zeta$ are marked Poisson PPs with sigma-finite intensity measures $\la, \mu$ and
mark kernels $K, L$, respectively, then the \Renyi divergence of order $\al > 0$ is given by
$\Ren_\al( \law(\xi) \| \law(\zeta) ) = \Tsa_\al( \la \otimes K \| \mu \otimes L)$.
\end{theorem}
\begin{proof}
Because $\law(\xi) = P_{\la \otimes K}$ and $\law(\zeta) = P_{\mu \otimes L}$,
the claim follows by Theorem~\ref{the:PPPRenyi}.
\end{proof}

Theorem \ref{the:MarkedPPPRenyi} combined with Theorem~\ref{the:TsallisKernel}
now allows one to compute Kullback--Leibler and \Renyi divergences of marked Poisson PPs in terms
of the intensity measures $\la,\mu$ and the kernels $K,L$.

\subsection{Compound Poisson processes}
\label{sec:CompoundPoissonProcesses}


A set of points $(t_i, x_i)$ associated with time stamps $t_i \in \R_+$ and labels $x_i \in \R^d$
can be modelled as a point pattern on $\R_+ \times \R^d$, or as a marked point pattern on $\R_+$ with
mark space $\R^d$.
Denote by $N(\R_+ \times \R^d)$ the set of point patterns $\eta$ on $\R_+ \times \R^d$ such that
$\eta([0,t] \times \R^d) < \infty$ for all $t \in \R_+$.  The \new{cumulative process} of such a point pattern is a function $X \colon \R_+ \to \R^d$ defined by
\begin{equation}
 \label{eq:CompoundProcess}
 X_t
 \weq \int_{[0,t] \times \R^d} x \, \eta(ds,dx).
\end{equation}
Denote by $\cum \colon \eta \mapsto X$ the map induced by \eqref{eq:CompoundProcess}.
A \new{compound Poisson process} with
event intensity measure $\la$ and increment probability kernel $K$
is the cumulative process
$X = \cum(\eta)$ of a point pattern $\eta$ sampled
from a Poisson PP distribution $P_{\la \otimes K}$ on $N(\R_+ \times \R^d)$
with intensity measure $(\la \otimes K)(dt,dx) = \la(dt) K_t(dx)$,
where $\la$ is a locally finite measure $\R_+$
and $K$ is a probability kernel from $\R_+$ into $\R^d$.
The law of the compound Poisson process is the pushforward measure $\law(X) = P_{\la \otimes K} \circ \cum^{-1}$.
The assumption that $\la$ is locally finite implies that
$P_{\la \otimes K}( N(\R_+ \times \R^d) ) = 1$, so that the right side of
\eqref{eq:CompoundProcess} is well defined for $P_{\la \otimes K}$-almost every $\eta$.

The map $\cum \colon \eta \mapsto X$ induced by \eqref{eq:CompoundProcess} maps
$N(\R_+ \times \R^d)$ into the set $F(\R_+, \R^d)$ of right-continuous and piecewise constant
functions from $\R_+$ into $\R^d$.
However, this map is not bijective because: (i)
two simultaneous events $(t,x)$ and $(t,y)$ have the same contribution as a single event $(t,x+y)$ to the cumulative process; and (ii) points $(t,0)$ with zero mark do not contribute anything.
To rule out such identifiability issues, we will restrict to a set $N_s(\R_+ \times \R^d)$ of point patterns $\eta \in N(\R_+ \times \R^d)$ such that
$\eta(\{t\} \times \R^d) \in \{0,1\}$ for all $t$, and
$\eta(\R_+ \times \{0\}) = 0$.
This is why we will assume that $\la$ is locally finite, diffuse in the sense that $\la(\{t\}) = 0$ for all $t$, and that
the probability kernel $K$ satisfies $K_t(\{0\}) = 0$ for all $t$.
Under these assumptions, it follows \cite[Proposition 6.9]{Last_Penrose_2018} that 
$P_{\la \otimes K}( N_s(\R_+ \times \R^d) ) =1$.

\begin{theorem}
\label{the:CompoundPoissonRenyi}
For compound Poisson processes
$X = (X_t)_{t \in \R_+}$ and $Y = (Y_t)_{t \in  \R_+}$
with diffuse locally finite event intensity measures $\la$ and $\mu$,
and increment probability kernels $K$ and $L$
such that $K_t(\{0\}) = 0$ and $L_t(\{0\}) = 0$,
the \Renyi divergence of order $\al > 0$ is given by
$\Ren_\al( \law(X) \| \law(Y) ) = \Tsa_\al( \la \otimes K \| \mu \otimes L)$.
\end{theorem}

\begin{proof}
Let us equip $F(\R_+, \R^d)$ with the sigma-algebra generated by the evaluation maps $X \mapsto X_t$.
The aggregation map $\cum \colon N(\R_+ \times \R^d) \to F(\R_+, \R^d)$
restricted to $N_s(\R_+ \times \R^d)$ is bijective, with inverse map given by
\[
 (\cum^{-1}(X))(C)
 \weq \#\{ (t,x) \in C \colon X_t - X_{t-} = x, \, x \ne 0 \},
\]
where $X_{t-} = \lim_{s \uparrow t} X_s$ for $t > 0$ and $X_{t-} = X_0$ for $t=0$.
Standard techniques (as in the proof of Theorem~\ref{the:PPPRenyi}) imply that
$\cum \colon N_s(\R_+ \times \R^d) \to F(\R_+, \R^d)$
is measurable with a measurable inverse. Because $P_{\la \otimes K}$
and $P_{\mu \otimes L}$ have all their mass supported on $N_s(\R_+ \times \R^d)$,
it follows (Lemma~\ref{the:RenyiBijection}) that
$\Ren_\al(\law(X) \| \law(Y)) = \Ren_\al(P_{\la \otimes K} \| P_{\mu \otimes L})$.
Theorem~\ref{the:PPPRenyi} 
then implies that $\Ren_\al(\law(X) \| \law(Y)) = \Tsa_\al( \la \| \mu )$.

\end{proof}

By combining Theorem \ref{the:CompoundPoissonRenyi} with Theorem~\ref{the:TsallisKernel}, 
we may compute information divergences of compound Poisson processes.
For example,
the \Renyi divergence of order $\al \notin \{0,1\}$
for compound Poisson processes $X = (X_t)_{t \in \R_+}$ and $Y = (Y_t)_{t \in \R_+}$
with event intensity measures $\la(dt) = f_t dt$ and $\mu(dt) = g_t dt$
and increment probability kernels $K$ and $L$
is given by
\[
 \begin{aligned}
 &\Ren_\al( \law(X) \| \law(Y) ) \\
 &\weq
 \Ren_\al(P_\la \| P_\mu) + \int_0^\infty \Tsa_\al(K_t \| L_t) \, f_t^\al g_t^{1-\al} \, dt.
 \end{aligned}
\]
This formula demonstrates how the information content decomposes into
two parts:
$\Ren_\al(P_\la \| P_\mu)$ associated with only observing the jump instants of the compound
Poisson processes, and the additional term
$\int_0^\infty \Tsa_\al(K_t \| L_t) \, f_t^\al g_t^{1-\al} \, dt$
characterising the information gain when we also observe the jump sizes.

\section{Proofs}
\label{sec:Proofs}

This section contains the proofs of the main results, with some of the technical parts postponed to the appendix.

\subsection{Proof of Theorem~\ref{the:Tsallis}}
\label{sec:ProofTsallis}

\begin{lemma}
\label{the:RenyiPoisson}
The \Renyi divergence of Poisson distributions $p_s$ and $p_t$ with means $s,t \in \R_+$ is given by\footnote{$\Ren_\al(p_s \| p_t) = \infty$ when $\al \ge 1$, $s>0$, and $t=0$ by our 0-division conventions.}
\begin{equation}
 \label{eq:RenyiPoisson}
 \Ren_\al(p_s \| p_t)
 \weq
 \begin{cases}
 1(s=0) t, &\quad \al = 0, \\
 \frac{\al s + (1-\al)t - s^\al t^{1-\al}}{1-\al}, &\quad \al \notin \{0,1\},\\
 s \log \frac{s}{t} + t - s, &\quad \al = 1.
 \end{cases}
\end{equation}
\end{lemma}
\begin{proof}
Assume first that $s,t > 0$.
By definition, the \Renyi divergence
or order $\al \notin \{0,1\}$ equals $\Ren_\al(p_s \| p_t) = \frac{1}{\al-1} \log Z_\al(p_s \| p_t)$, where 
\begin{align*}
 Z_\al( p_s \| p_t )
 &\weq \sum_{x=0}^\infty \Big( e^{-s} \frac{s^x}{x!} \Big)^\al
 \Big( e^{-t} \frac{t^x}{x!} \Big)^{1-\al} \\
 &\weq \exp\Big( - \Big( \al s + (1-\al)t - s^\al t^{1-\al} \Big) \Big).
\end{align*}
By taking logarithms, \eqref{eq:RenyiPoisson} follows for $\al \notin\{0,1\}$.
We also note that $\log \frac{e^{-s} \frac{s^x}{x!}}{e^{-t} \frac{t^x}{x!}} = x \log\frac{s}{t} + t-s$, 
and taking expectations with respect to $p_s$ yields
\eqref{eq:RenyiPoisson} for $\al=1$. Furthermore, when $s,t > 0$, both $p_s$ and $p_t$ assign
strictly positive probabilities to all nonnegative integers. Therefore,
$\Ren_0(p_s \| p_t ) = - \log p_t( \Z_+ ) = 0$, confirming \eqref{eq:RenyiPoisson} for $\al=0$.

We also note that the Poisson distribution with mean 0 equals the Dirac measure at 0.
Therefore, a simple computation shows that for all $s,t>0$,
\[
 \Ren_\al(p_s \| \delta_0)
 \weq
 \begin{cases}
  0, &\quad \al = 0, \\
  \frac{\al}{1-\al} s, &\quad \al \in (0,1), \\
  \infty, &\quad \al \ge 1,
 \end{cases}
\]
and $\Ren_\al(\delta_0 \| p_t) = t$ for all $\al \in [0,\infty)$. Finally, $\Ren_\al(\delta_0 \| \delta_0) = 0$ for all $\al \in [0,\infty)$.  By recalling our conventions with dividing by 0, we may conclude that \eqref{eq:RenyiPoisson} holds for all $s,t,\al \in [0,\infty)$.
\end{proof}

\begin{proof}[Proof of Theorem~\ref{the:Tsallis}]
Let $\la,\mu$ be sigma-finite measures admitting densities $f,g \colon S \to \R_+$ with respect to a sigma-finite measure $\nu$ on a measurable space $(S,\cS)$.
Let $p_{f(x)}$ and $p_{g(x)}$ be Poisson distributions with means $f(x)$ and $g(x)$.
By applying the \Renyi divergence formula for Poisson distributions in Lemma~\ref{the:RenyiPoisson}, we find that
\begin{align*}
 &\int_S \Ren_\al(p_{f(x)} \| p_{g(x)}) \, \nu(dx) \\
 &\quad \weq
 \begin{cases}
 \int_S 1(f=0) g \, d\nu, &\quad \al = 0, \\
 \int_S \frac{\al f + (1-\al) g - f^\al g^{1-\al}}{1-\al} \, d\nu, &\quad \al \notin \{0,1\},\\
 \int_S \left( f \log \frac{f}{g} + g - f \right) \, d\nu, &\quad \al = 1.
 \end{cases}
\end{align*}
By comparing this with the definition of the Tsallis divergence \eqref{eq:Tsallis}, we find that 
\begin{equation}
 \label{eq:TsallisPoisson}
 \Tsa_\al(\la \| \mu)
 \weq \int_S \Ren_\al(p_{f(x)} \| p_{g(x)}) \, \nu(dx).
\end{equation}
Because \Renyi divergences of probability measures are nonnegative, 
\eqref{eq:TsallisPoisson} shows that $\Tsa_\al(\la \| \mu)$ is a well-defined element in $[0,\infty]$ for all $\al \in \R_+$.
We also know \cite[Theorem 3]{VanErven_Harremoes_2014} that \Renyi divergences of probability measures are nondecreasing in $\al$, so it also follows from \eqref{eq:TsallisPoisson} that $\Tsa_\al(\la \| \mu)$ is
nondecreasing in $\al$.


Let us next verify that $\al \mapsto T_\al(\la \| \mu)$ is continuous on $A = \{\al \in \R_+ \colon \Tsa_\al(\la \| \mu) < \infty\}$.
To this end, 
assume that $A$ is nonempty and $\al_n \to \al$ for some $\al_n,\al \in A$.
Because $\Tsa_\al(\la \| \mu)$ is nondecreasing in $\al$, we see that $A$ is an interval, and that there exists a number $\beta \in A$ such that $\al_n,\al \le \beta$ for all $n$.
Denote 
$r_a(x) = \Ren_a( p_{f(x)} \| p_{g(x)} )$ for $a,x \in \R_+$,
%
and let $U = \{x \in S \colon r_\beta(x) < \infty\}$.
Because $\Tsa_\beta(\la \| \mu) = \int_S r_\beta \, d\nu$ is finite, it follows that $\nu(U^c) = 0$.
Therefore,
\[
 \Tsa_{\al_n}(\la \| \mu)
 \weq \int_S r_{\al_n} \, d\nu
 \weq \int_U r_{\al_n} \, d\nu.
\]
For any $x \in U$, the monotonicity of \Renyi divergences \cite[Theorem 3]{VanErven_Harremoes_2014} implies that $r_\al(x), r_{\al_n}(x) \le r_\beta(x)$.  In particular, $r_\al(x), r_{\al_n}(x)$ are finite for $x \in U$.
The continuity of \Renyi divergences \cite[Theorem 7]{VanErven_Harremoes_2014} then implies that
$r_{\al_n} \to r_\al$ pointwise on $U$. Lebesgue's dominated convergence theorem then implies that
\[
 \Tsa_{\al_n}(\la \| \mu)
 \weq \int_U r_{\al_n} \, d\nu
 \wto \int_U r_{\al} \, d\nu
 \weq \Tsa_{\al}(\la \| \mu).
\]

Finally, let us verify that the right side of \eqref{eq:Tsallis} does not depend on the choice of the densities nor the reference measure.  Assume that $\la,\mu$ admit densities $f_1, g_1 \colon S \to \R_+$ with respect to a sigma-finite measure $\nu_1$, and densities $f_2, g_2 \colon S \to \R_+$ with respect to a sigma-finite measure $\nu_2$.
Define $\nu = \nu_1+\nu_2$.  Then $\nu_1,\nu_2 \ll \nu$ and $\nu$ is sigma-finite.
The Radon--Nikodym theorem \cite[Theorem 2.10]{Kallenberg_2002} implies that there exist densities $h_1,h_2 \colon S \to \R_+$ of $\nu_1,\nu_2$ with respect to $\nu$.
Then
\[
 \la(A)
 \weq \int_A f_i \, d\nu_i
 \weq \int_A \tilde f_i \, d\nu
 \quad \text{for $i=1,2$},
\]
where $\tilde f_i = f_i h_i$.  We see that both $\tilde f_1$ and $\tilde f_2$ are densities of $\la$ with respect to $\nu$.
The Radon--Nikodym theorem \cite[Theorem 2.10]{Kallenberg_2002} implies that
$\tilde f_1 = \tilde f_2$ $\nu$-almost everywhere. Similarly, we see that
both $\tilde g_1$ and $\tilde g_2$ are densities of $\mu$ with respect to $\nu$, so that
$\tilde g_1 = \tilde g_2$ $\nu$-almost everywhere. 
Formula \eqref{eq:RenyiPoisson} shows that 
\Renyi divergences of Poisson distributions are homogeneous in the sense that 
$\Ren_\al(p_{c s} \| p_{c t}) = c \Ren_\al(p_{s} \| p_{t})$ for all $c \in \R_+$.
As a consequence, we see that
\begin{align*}
 \int_S \Ren_\al(p_{\tilde f_1(x)} \| p_{\tilde g_1(x)}) \, \nu(dx)
 &\weq \int_S \Ren_\al(p_{f_1(x)} \| p_{g_1(x)}) \, \nu_1(dx), \\
 \int_S \Ren_\al(p_{\tilde f_2(x)} \| p_{\tilde g_2(x)}) \, \nu(dx)
 &\weq \int_S \Ren_\al(p_{f_2(x)} \| p_{g_2(x)}) \, \nu_2(dx).
\end{align*}
Because $\tilde f_1 = \tilde f_2$ and $\tilde g_1 = \tilde g_2$ $\nu$-almost everywhere, 
it follows that all of the above integrals are equal to each other.  In particular,
\[
 \int_S \Ren_\al(p_{f_1(x)} \| p_{g_1(x)}) \, \nu_1(dx)
 \weq \int_S \Ren_\al(p_{f_2(x)} \| p_{g_2(x)}) \, \nu_2(dx).
\]
In light of \eqref{eq:TsallisPoisson}, we conclude that the value of $\Tsa_\al( \la \| \mu)$
as defined by formula \eqref{eq:TsallisPoisson} is the same for 
both triples $(f_1,g_1,\nu_1)$ and $(f_2, g_2, \nu_2)$.

\end{proof}

\subsection{Proof of Theorem~\ref{the:TsallisKernel}}
\label{sec:TsallisKernelProof}
\begin{proof}[Proof of Theorem~\ref{the:TsallisKernel}]
Because $\la(dt) = f_t \nu(dt)$ and $K_t(dx) = k_t(x) M_t(dx)$, we see that
\begin{align*}
 (\la \otimes K)(C)
 &\weq \int_{S_1} \left( \int_{S_2} 1_C(t,x) \, K_t(dx) \right) \la(dt) \\
 &\weq \int_{S_1} \left( \int_{S_2} 1_C(t,x) \, k_t(x) \, M_t(dx) \right) f_t \nu(dt) \\
 &\weq \int_{S_1} \left( \int_{S_2} 1_C(t,x) \, f_t k_t(x) \, M_t(dx) \right) \nu(dt) \\
 &\weq \int_C f_t k_t(x) \, (\nu \otimes M)(dt,dx)
\end{align*}
for all measurable $C \subset S_1 \times S_2$, so that $(t,x) \mapsto f_t k_t(x)$ is a density of $\la \otimes K$ with respect to $\nu \otimes M$.
Similarly, $t \mapsto g_t \ell_t(x)$ is a density
of $\mu \otimes L$ with respect to $\nu \otimes M$.

By \eqref{eq:Tsallis}, the Tsallis divergence of order $\al \ne 0,1$ is given by
\begin{equation}
 \label{eq:TsallisKernel1}
 \Tsa_\al( \la \otimes K \| \mu \otimes L )
 \weq \int_{S_1} \left( \int_{S_2} \tau_\al(t,x) \, M_t(dx) \right) \, \nu(dt),
\end{equation}
where the integrand equals
\[
 \tau_\al(t,x)
 \weq \frac{ \al f_t k_t(x) + (1-\al) g_t \ell_t(x)
 - f_t^\al g_t^{1-\al} k_t(x)^\al \ell_t(x)^{1-\al} }{1-\al}.
\]
The integrand may also be written as
\begin{align*}
 \tau_\al(t,x)
 &\weq \frac{ \al f_t k_t(x) + (1-\al) g_t \ell_t(x)}{1-\al} \\
 &\qquad - \left( \frac{ \al k_t(x) + (1-\al) \ell_t(x) }{1-\al} \right) f_t^\al g_t^{1-\al} \\
 &\qquad + \left( \frac{\al k_t(x) + (1-\al) \ell_t(x) - k_t(x)^\al \ell_t(x)^{1-\al} }{1-\al} \right) f_t^\al g_t^{1-\al}.
\end{align*}
We note that
$\int_{S_2} k_t(x) \, M_t(dx) = 1 $ and 
$\int_{S_2} \ell_t(x) \, M_t(dx) = 1 $, and that
\[
 \int_{S_2}\left( \frac{\al k_t(x) + (1-\al) \ell_t(x) - k_t(x)^\al \ell_t(x)^{1-\al} }{1-\al} \right) M_t(dx)
 \weq \Tsa_\al(K_t \| L_t).
\]
It follows that the inner integral in \eqref{eq:TsallisKernel1} equals
\begin{align*}
 \int_{S_2} \tau_\al(t,x) \, M_t(dx)
 &\weq \frac{ \al f_t + (1-\al) g_t - f_t^\al g_t^{1-\al} }{1-\al} 
 + \Tsa_\al(K_t \| L_t) \, f_t^\al g_t^{1-\al}.
\end{align*}
By integrating both sides against $\nu(dt)$, we obtain \eqref{eq:TsallisKernel} for $\al \notin\{0,1\}$.


Let us now consider the case with $\al = 1$.
In this case $\Tsa_\al( \la \otimes K \| \mu \otimes L )$ is again given by \eqref{eq:TsallisKernel1},
but now we replace $\tau_\al$ by
\begin{align*}
 \tau_1(t,x)
 &\weq f_t k_t(x) \log \frac{f_t k_t(x)}{g_t \ell_t(x)} + g_t \ell_t(x) - f_t k_t(x) \\
 &\weq f_t k_t(x) \log \frac{f_t}{g_t} + f_t k_t(x) \log \frac{k_t(x)}{\ell_t(x)} + g_t \ell_t(x) - f_t k_t(x).
\end{align*}
By integrating the above equation against $M_t(dx)$, we find that
\begin{align*}
 \int_{S_2} \tau_1(t,x) \, M_t(dx)
 &\weq f_t \log \frac{f_t}{g_t} + g_t  - f_t + f_t \Tsa_1( K_t \| L_t ).
\end{align*}
By further integrating this against $\nu(dt)$, it follows that
\begin{align*}
 \int_{S_1} \left( \int_{S_2} \tau_1(t,x) \, M_t(dx) \right) \nu(dt)
 \weq \Tsa_1(\la \| \mu) + \int_{S_1} \Tsa_1( K_t \| L_t ) \, \la(dt),
\end{align*}
from which we conclude the validity of \eqref{eq:TsallisKernel} for $\al = 1$.

Finally, for $\al = 0$ we note that
\begin{align*}
 &\{ (t,x) \colon f_t k_t(x) = 0\} \\
 &\weq (\{f=0\} \times S_2) \cup \{(t,x) \colon f_t \ne 0, \, k_t(x) = 0\}.
\end{align*}
Hence by \eqref{eq:Tsallis},
\begin{align*}
 T_0(\la \otimes K \| \mu\otimes L)
 &\weq (\mu \otimes L)\{ (t,x) \colon f_t k_t(x) = 0\} \\
 &\weq \mu\{f=0\} + \int_{f \ne 0} L_t\{ k_t = 0\} \, \mu(dt) \\
 &\weq T_0(\la \| \mu) + \int_{f \ne 0} \Tsa_0(K_t \| L_t) \, \mu(dt).
\end{align*}
\end{proof}

\subsection{Proof of Theorem~\ref{the:PoissonDensityFinite}}
\label{sec:PoissonDensityFiniteProof}

The following result shows that embeddings $(x_1,\dots,x_n) \mapsto \sum_{i=1}^n \delta_{x_i}$ are measurable even when singleton sets in $(S,\cS)$ might be nonmeasurable. This is the reason why there is no need to deal with `chunks' as in \cite{Brown_1971}, and we get a conceptually simplified proof of Theorem~\ref{the:PoissonDensityFinite}.

\begin{lemma}
\label{the:EmbeddingMeasurable}
For any $n \ge 1$, the function $\iota_n \colon S^n \to N(S)$ by
$\iota_k(x_1,\dots,x_n) = \sum_{i=1}^n \delta_{x_i}$ is measurable.
\end{lemma}
\begin{proof}
Fix a set $A \in \cS$ and an integer $k \ge 0$. Let $C = \{\eta \colon \eta(A)=k\}$.
Then the set 
\begin{align*}
 \iota_n^{-1}(C)
 &\weq \Big\{(x_1,\dots,x_n) \in S^n \colon \sum_{i=1}^n \delta_{x_i}(A) = k \Big\}
\end{align*}
consists of the $n$-tuples in $S$ for which
exactly $k$ members belong to $A$, and the remaining members belong to $A^c$.
The set of all such $n$-tuples can be written as
\[
 \iota_n^{-1}(C)
 \weq \nhquad \bigcup_{ b \in \{0,1\}^n: \sum b_i = k }
 \nhquad 1_A^{-1}(\{b_1\}) \times \cdots \times 1_A^{-1}(\{b_n\}),
\]
where $1_A^{-1}(\{b_k\})$ denotes the preimage of the indicator
function $1_A \colon S \to \{0,1\}$ for $\{b_k\}$,
and equals $A$ for $b_k=1$ and $A^c$ for $b_k=0$.
Because all sets appearing in the finite union on the right are products of $A$ and $A^c$,
we conclude that $\iota_n^{-1}(C) \in \cS^{\otimes n}$.
Because $\cN(S)$ is the sigma-algebra generated by sets of form $C$,
we conclude that $\iota_n^{-1}(C) \in \cS^{\otimes n}$ for all $C \in \cN(S)$.
\end{proof}

\begin{proof}[Proof of Theorem~\ref{the:PoissonDensityFinite}]
Assume that
$P_\la, P_\mu$ are Poisson PP distributions with finite intensity measures $\la,\mu$ such that $\la \ll \mu$.
To rule out trivialities, we assume that $\la,\mu$ are nonzero.
The Poisson PP distributions may \cite[Proposition 3.5]{Last_Penrose_2018} then be represented as
\begin{equation}
 \label{eq:FinitePoissonRepresentation}
 P_\la \weq \sum_{n \ge 0}e^{-\la(S)} \frac{\la(S)^n}{n!} \, \la_1^{\otimes n} \!\circ\! \iota_n^{-1},
 \qquad
 P_\mu \weq \sum_{n \ge 0} e^{-\mu(S)} \frac{\mu(S)^n}{n!} \, \mu_1^{\otimes n} \!\circ\! \iota_n^{-1},
\end{equation}
where
$\la_1^{\otimes n}, \mu_1^{\otimes n}$ are $n$-fold products of probability measures $\la_1 = \la/\la(S)$ and $\mu_1 = \mu/\mu(S)$ on $S$, and
$\iota_n(x_1,\dots, x_n) = \sum_{i=1}^n \delta_{x_i}$.
(Lemma~\ref{the:EmbeddingMeasurable} guarantees that the maps $\iota_n \colon S^n \to N(S)$ are measurable,
and therefore $P_\la, P_\mu$ are well-defined probability measures on $N(S)$.)

Recall that
$\phi \colon S \to \R_+$ is a density of $\la$ with respect to $\mu$, and 
consider the function $\Phi \colon N(S) \to \R_+$ such that
\begin{equation}
 \label{eq:PoissonDensityAlt}
 \Phi(\eta)
 = e^{\mu(S)-\la(S)} \prod_{i=1}^n \phi(x_i)
 \qquad \text{for} \quad \eta = \sum_{i=1}^n \delta_{x_i},
\end{equation}
and $\Phi(\eta) = 0$ for $\eta(S) = \infty$. We will show that $\Phi$ is a density of $P_\la$ with respect to $P_\mu$.
To do this, fix a measurable set $C \subset N(S)$, and note that
\begin{align*}
 \int_C \Phi(\eta) \, P_\mu(d\eta)
 &\weq \sum_{n \ge 0} e^{-\mu(S)} \frac{\mu(S)^n}{n!}
 \int_C \Phi(\eta) \, \mu_1^{\otimes n} \!\circ\! \iota_n^{-1}(d\eta) \\
 &\weq \sum_{n \ge 0} e^{-\mu(S)} \frac{\mu(S)^n}{n!}
 \int_{\iota_n^{-1}(C)} \Phi(\iota_n(x)) \, \mu_1^{\otimes n} (dx) \\
 &\weq \sum_{n \ge 0} e^{-\la(S)} \frac{\mu(S)^n}{n!}
 \int_{\iota_n^{-1}(C)} \phi^{\otimes n}(x) \, \mu_1^{\otimes n} (dx),
\end{align*}
where $\phi^{\otimes n}(x_1,\dots,x_n)= \prod_{i=1}^n \phi(x_i)$.
We also note that $\frac{\mu(S)}{\la(S)} \phi$ is a density of $\la_1$ with respect to $\mu_1$, and 
therefore, $(\frac{\mu(S)}{\la(S)})^n \phi^{\otimes n}$
is a density of $\la_1^{\otimes n}$ with respect to $\mu_1^{\otimes n}$.
Then
\[
 \la_1^{\otimes n}(\iota_n^{-1}(C))
 \weq \int_{\iota_n^{-1}(C)} \left(\frac{\mu(S)}{\la(S)} \right)^n \phi^{\otimes n}(x) \, \mu_1^{\otimes n}(dx),
\]
and it follows that
\begin{align*}
 \int_C \Phi(\eta) \, P_\mu(d\eta)
 &\weq \sum_{n \ge 0} e^{-\la(S)} \frac{\mu(S)^n}{n!}
 \left(\frac{\la(S)}{\mu(S)} \right)^n \la_1^{\otimes n}(\iota_n^{-1}(C)) \\
 &\weq \sum_{n \ge 0} e^{-\la(S)} \frac{\la(S)^n}{n!}
 \la_1^{\otimes n}(\iota_n^{-1}(C)).
\end{align*}
In light of \eqref{eq:FinitePoissonRepresentation}, we conclude that
$\int_C \Phi(\eta) \, P_\mu(d\eta) = P_\la(C)$. Hence $\Phi$ is a density of $P_\la$
with respect to $P_\mu$.
In particular, $P_\la \ll P_\mu$.


Finally, let us verify that the function $\Phi$ can be written in form~\eqref{eq:PoissonDensityFinite}.
By definition~\eqref{eq:PoissonDensityAlt}, we see that $\Phi(\eta) = 0$ whenever $\eta(S)=\infty$ or $\eta\{\phi=0\} > 0$.  Hence
$\Phi$ vanishes outside the set $M \cap \Omega$ where
$M = \{\eta \in N(S) \colon\eta\{\phi=0\} = 0 \}$
and 
$\Omega = \{\eta \in N(S) \colon \eta(S) < \infty \}$.
On the other hand,
$\prod_{i=1}^n \phi(x_i) = \exp(\int_S \log\phi \, d\eta)$ for every point pattern in $M \cap \Omega$ of form $\eta = \sum_{i=1}^n \delta_{x_i}$. By noting that $\mu(S) - \la(S) = \int_S (1 - \phi) \, d\mu$, we see that $\Phi$ may be written as
in \eqref{eq:PoissonDensityFinite}, but with $1_M$ replaced by $1_{M \cap \Omega}$.
Because Campbell's theorem implies that $P_\mu(\Omega)=1$, we see that $\Phi$ is equal to the function in \eqref{eq:PoissonDensityFinite} as an element of $L_1(N(S), \cN(S), P_\mu)$.
\end{proof}

\subsection{Proof of Theorem~\ref{the:PoissonDensity}}
\label{sec:PoissonDensitySigmafiniteProof}

We start by proving the following simple upper bound that confirms that $\la$ is finite whenever
$\la \ll \mu$ and $H(\la,\mu) < \infty$ for some finite measure $\mu$.

\begin{lemma}
\label{the:LaBound}
Assume that $\la \ll \mu$ and that $\phi \colon S \to \R_+$ is a density of $\la$ with respect to $\mu$.
Then $\la(B) \le 4 \mu(B) + 3 \int_B (\sqrt{\phi}-1)^2 \, d\mu$ for any measurable $B \subset S$.
Especially, $\la(S) \le 4 \mu(S) + 6 H(\la,\mu)^2$.
\end{lemma}
\begin{proof}
By writing $\phi = 1 + (\phi-1) \le 1 + (\phi-1)_+$, we see that $\la(B) =  \int_B \phi \, d\mu$ is bounded by
\begin{equation}
 \label{eq:LaBound1}
 \la(B)
 \wle \mu(B) + \int_B (\phi-1)_+ \, d\mu.
\end{equation}
We also note that $t \mapsto \frac{t-1}{t+1}$ is increasing on $\R_+$, so that 
\[
 (\sqrt{\phi}-1)^2
 \weq \frac{\sqrt{\phi}-1}{\sqrt{\phi}+1} (\phi-1)
 \wge \frac13 (\phi-1)
 \qquad \text{for $\phi > 4$}.
\]
Because $(\phi-1)_+ \le 3$ for $\phi \le 4$, it follows that
\begin{align*}
 \intlim_B (\phi-1)_+ \, d\mu
 &\weq \nhquad \nhquad \intlim_{B \cap \{\phi \le 4\}} \nhquad (\phi-1)_+ \, d\mu
 \ + \nhquad \intlim_{B \cap \{\phi > 4\}} \nhquad (\phi-1)_+ \, d\mu \\
 &\wle 3 \mu(B) + 3 \intlim_B (\sqrt{\phi}-1)^2 \, d\mu.
\end{align*}
The first claim follows by combining this with \eqref{eq:LaBound1}.
The second claim follows by noting that 
$\int_S (\sqrt{\phi}-1)^2 \, d\mu = 2 H(\la,\mu)^2$ due to Proposition~\ref{the:Hellinger}.
\end{proof}

\begin{proof}[Proof of Theorem~\ref{the:PoissonDensity}]
Consider sigma-finite intensity measures such that $\la \ll \mu$ and $H(\la,\mu) < \infty$.
Fix a density $\phi = \frac{d\la}{d\mu}$, and
a sequence $S_n \uparrow S$ such that $\mu(S_n) < \infty$ for all $n$.

(i) \emph{Computing a truncated density}.
Let $P_{\la_n}$ and $P_{\mu_n}$ be Poisson PP distributions with truncated intensity
measures $\la_n(B) = \la(B \cap S_n)$ and $\mu_n(B) = \mu(B \cap S_n)$ on $(S,\cS)$.
We also employ the same notation $\eta_n(B) = \eta(B \cap S_n)$ for truncations of point patterns $\eta \in N(S)$.
We note that $\la_n \ll \mu_n$, and that $\phi$ serves also as a density of $\la_n$ with respect to $\mu_n$.
Our choice of $S_n$ implies that $\mu_n$ is a finite measure.
We also note that $\la_n$ is finite because $\la_n(S) \le 4 \mu_n(S) + 3 \int_{S} (\sqrt{\phi}-1)^2 \, d\mu_n
= 4 \mu(S_n) + 6 H^2(\la,\mu)$ due to Lemma~\ref{the:LaBound}.
Theorem~\ref{the:PoissonDensityFinite} now implies that
$P_{\la_n} \ll P_{\mu_n}$, with a likelihood ratio given by
\begin{equation}
 \label{eq:PoissonDensityTruncated}
 \Phi_n(\eta)
 \weq 1_M(\eta) \, \exp\bigg( \int_S ( 1 - \phi) \, d\mu_n + \int_S \log \phi \, d\eta \bigg),
\end{equation}
where
\begin{equation}
 \label{eq:M}
 M \weq \{ \eta \in N(S) \colon \eta\{\phi=0\} = 0\}.
\end{equation}

(ii) \emph{Approximate Laplace functional.}
Fix a measurable function $u \colon S \to \R_+$.
Monotone convergence of integrals then implies that
$\eta_n(u) = \eta( u 1_{S_n} ) \uparrow \eta(u)$ for all $\eta \in N(S)$.
Lebesgue's dominated convergence theorem
then implies that
\begin{equation}
 \label{eq:PoissonLaplaceTruncatedLimit}
 \int_{N(S)} e^{-\eta(u)} P_\la(d\eta)
 \weq \lim_{n \to \infty} \int_{N(S)} e^{-\eta_n(u)} P_\la(d\eta).
\end{equation}
By \cite[Theorem 5.2]{Last_Penrose_2018},
\begin{align*}
 \int_{N(S)} e^{-\eta_n(u)} P_\la(d\eta)
 &\weq \int_{N(S)} e^{-\eta(u)} P_{\la_n} (d\eta) \\
 &\weq \int_{N(S)} e^{-\eta(u)} \Phi_n(\eta) P_{\mu_n} (d\eta) \\
 &\weq \int_{N(S)} e^{-\eta_n(u)} \Phi_n(\eta_n) P_{\mu} (d\eta).
\end{align*}
Together with \eqref{eq:PoissonLaplaceTruncatedLimit},
we conclude that
\begin{equation}
 \label{eq:PoissonDensityTruncatedLimit}
 \int_{N(S)} e^{-\eta(u)} \, P_\la(d\eta)
 \weq \lim_{n \to \infty} \int_{N(S)} e^{-\eta_n(u)} \Phi_n(\eta_n) \, P_{\mu} (d\eta).
\end{equation}

(iii) \emph{Identifying the limiting density.}
Let us next identify the limit of $\Phi_n(\eta_n)$ as $n \to \infty$.
In light of \eqref{eq:PoissonDensityTruncated}, we see that
\begin{equation}
 \label{eq:PoissonDensityPrelimit}
 \Phi_n(\eta_n)
 \weq 1_{M}(\eta_n) \exp\bigg( \int_{S_n} ( 1 - \phi) \, d\mu + \int_{S_n} \log \phi \, d\eta \bigg).
\end{equation}
Even though $S_n \uparrow S$, the integrals on the right side above may not converge as expected because
$\int_S ( 1 - \phi) \, d\mu$ and $\int_S \log \phi \, d\eta$ are not necessarily well defined.  Also, the compensated
integral $\int_S \log\phi \, d(\eta-\mu)$ might diverge.
A key observation (proven soon) is that the compensated integral $\int_A \log\phi \, d(\eta-\mu)$ will converge for $P_\mu$-almost every $\eta$,
where
\[
 A
 \weq \{x \in S \colon \abs{\log\phi(x)} \le 1 \}.
\]
With this target in mind, we will reorganise the integral terms of \eqref{eq:PoissonDensityPrelimit} according to
\begin{equation}
 \label{eq:PoissonDensitySplit}
 \int_{S_n} ( 1 - \phi) \, d\mu + \int_{S_n} \log \phi \, d\eta
 \weq W_n(\eta) + Z_n(\eta) + w_n + z_n,
\end{equation}
where
\[
\begin{aligned}
 W_n(\eta) &= \int_{A \cap S_n} \nhquad \log\phi \, d\eta - \int_{A \cap S_n} \nhquad \log\phi \, d\mu, \\
 Z_n(\eta) &= \int_{A^c \cap S_n} \nhquad \log\phi \, d\eta,
\end{aligned}
 \qquad
\begin{aligned}
 w_n &= \int_{A \cap S_n} (\log\phi + 1 - \phi) \, d\mu, \\
 z_n &= \int_{A^c \cap S_n} (1 - \phi) \, d\mu.
\end{aligned}
\]

We will show that all terms on the right side of \eqref{eq:PoissonDensitySplit} converge for all $\eta \in M \cap \Omega$, where $M$ is defined by \eqref{eq:M} and $\Omega = \Omega_1 \cap \Omega_2$, where
\[
 \Omega_1
 \weq \left\{ \eta \in N(S) \colon \int_{A^c \cap \{\phi > 0\}} \nhquad \abs{\log \phi} \, d\eta < \infty, \ \eta(A^c) < \infty \right\}
\]
and
\[
 \Omega_2
 \weq \left\{ \eta \in N(S) \colon \text{$\int_A \log \phi \, d(\eta-\mu)$ converges}, \
 \eta(S_n) < \infty \ \text{for all $n$} \right\}.
\]
First, the functions $1_{A^c \cap S_n} \log \phi$ are dominated in absolute value by $1_{A^c} \abs{\log \phi}$
for all $n$. The dominating function is integrable with respect to any $\eta \in M \cap \Omega$
by the definition of $\Omega_1$.
Lebesgue's dominated convergence theorem then implies that
\begin{equation}
 \label{eq:PoissonDensitySplit1}
 Z_n(\eta) \to \int_{A^c} \log \phi \, d\eta
\end{equation}
for every $\eta \in M \cap \Omega_1$.
The definition of $\Omega_2$ in turn implies that
\begin{equation}
 \label{eq:PoissonDensitySplit2}
 W_n(\eta) \to \int_A \log\phi \, d(\eta-\mu)
\end{equation}
for all $\eta \in \Omega_2$.
We also note that the functions associated with the definitions of $z_n$ and $w_n$ converge
pointwise according to
\begin{align*}
 (1-\phi) 1_{A^c \cap S_n} &\to (1-\phi) 1_{A^c}, \\
 (\log\phi + 1 - \phi) 1_{A \cap S_n} &\to (\log\phi + 1 - \phi) 1_{A}.
\end{align*}
By \eqref{eq:SquareRootDiff}, $\phi+1 \wle \frac{e + 1}{(e^{1/2}-1)^2} (\sqrt{\phi}-1)^2 $
on $A^c$, so that the functions $(1-\phi) 1_{A^c \cap S_n}$ are dominated in absolute value
by $\frac{e + 1}{(e^{1/2}-1)^2} (\sqrt{\phi}-1)^2$.
Similarly, by \eqref{eq:A2}, the functions $(\log\phi + 1 - \phi) 1_{A \cap S_n}$ are dominated in absolute value
by $2 e^3 (\sqrt{\phi}-1)^2$.
Both dominating functions are integrable due to
$\int_S (\sqrt{\phi}-1)^2 \, d\mu = 2 H^2(\la,\mu) < \infty$ (recall Proposition~\ref{the:Hellinger}).
Therefore, by dominated convergence, we see that
\begin{align}
 \label{eq:PoissonDensitySplit3}
 z_n &\wto \int_{A^c} ( 1 - \phi) \, d\mu, \\
 \label{eq:PoissonDensitySplit4}
 w_n &\wto \int_{A} (\log\phi + 1 - \phi) \, d\mu.
\end{align}
By plugging \eqref{eq:PoissonDensitySplit1}--\eqref{eq:PoissonDensitySplit4} into \eqref{eq:PoissonDensitySplit}, we find that for all $\eta \in M \cap \Omega$,
\[
 \lim_{n \to \infty} \left( \int_{S_n} ( 1 - \phi) \, d\mu + \int_{S_n} \log \phi \, d\eta \right)
 \weq \ell(\eta)
\]
where 
\begin{equation}
 \label{eq:PoissonLogDensityProof}
 \begin{aligned}
 \ell(\eta)
 &\weq\int_{A^c} \log \phi \, d\eta
 + \int_A \log\phi \, d(\eta-\mu) \\
 & \qquad + \int_{A^c} (1 - \phi) \, d\mu + \int_{A} (\log\phi + 1 - \phi) \, d\mu.
 \end{aligned}
\end{equation}
By noting that $\eta_n \in M$ whenever $\eta \in M \cap \Omega$,
we see in light of \eqref{eq:PoissonDensityPrelimit} that
\[
 \lim_{n \to \infty} \Phi_n(\eta_n) 
 \weq e^{\ell(\eta)}
\]
for all $\eta \in M \cap \Omega$. 
Furthermore, for any $\eta \in M^c \cap \Omega$, we note that $\eta\{ \phi =0 \} \ge 1$, and
the fact that $\{\phi=0\} \cap S_n \uparrow \{\phi=0\}$ then implies that $\eta(\{\phi=0\} \cap S_n) \ge 1$ eventually
for all large $n$. Hence for any $\eta \in M^c \cap \Omega$, $1_{M}(\eta_n) = 0$ eventually for all sufficiently large values of $n$.
We also note that $\eta_n(u) = \eta( u 1_{S_n} ) \uparrow \eta(u)$ by monotone convergence of integrals.  By denoting $\Phi(\eta) = 1_{M \cap \Omega}(\eta) e^{\ell(\eta)}$, we see that 
\begin{equation}
 \label{eq:PoissonDensityTruncatedPointwiseLimit}
 \lim_{n \to \infty} 1_\Omega(\eta) e^{-\eta_n(u)} \Phi_n(\eta_n)
 \weq e^{-\eta(u)} \Phi(\eta)
 \qquad \text{for all $\eta$}.
\end{equation}

(iv) \emph{Exchanging the limit and integral.}
Let us justify that we may interchange the limit and the integral in \eqref{eq:PoissonDensityTruncatedLimit}.
We know by Lemma~\ref{the:PoissonDensityIntegral} that $P_\mu(\Omega)=1$.
Therefore, \eqref{eq:PoissonDensityTruncatedLimit}
can be written as
\begin{equation}
 \label{eq:PoissonDensityTruncatedLimit2}
 \int_{N(S)} e^{-\eta(u)} \, P_\la(d\eta)
 \weq \lim_{n \to \infty} \int_{N(S)} 1_\Omega(\eta) e^{-\eta_n(u)} \Phi_n(\eta_n) \, P_{\mu} (d\eta).
\end{equation}
We wish to take the limit inside the integral in \eqref{eq:PoissonDensityTruncatedLimit2}.
To justify this, we note that the functions $f_n = 1_\Omega(\eta) e^{-\eta_n(u)} \Phi_n(\eta_n)$ are bounded by
$0 \le f_n \le g_n$, where $g_n = \Phi_n(\eta_n)$.  We also note \cite[Theorem 5.2]{Last_Penrose_2018},  that
\[
 \int_{N(S)} g_n \, dP_\mu
 \weq \int_{N(S)} \Phi_n(\eta) \, P_{\mu_n}(d\eta)
 \weq \int_{N(S)} P_{\la_n}(d\eta)
 \weq 1,
\]
because $\Phi_n = \frac{dP_{\la_n}}{dP_{\mu_n}}$.
Especially, $\abs{f_n} \le g_n$ for all $n$, and $\sup_n \int_{N(S)} g_n \, dP_\mu < \infty$.
A modified version of Lebesgue's dominated convergence theorem (Lemma~\ref{the:DOM}) then
justifies exchanging the limit and integral on the right side of \eqref{eq:PoissonDensityTruncatedLimit2},
and plugging in the limit of \eqref{eq:PoissonDensityTruncatedPointwiseLimit}
shows that
for all measurable $u \colon S \to \R_+$,
\begin{equation}
 \label{eq:PoissonDensityLaplace}
 \int_{N(S)} e^{-\eta(u)} \, P_\la(d\eta)
 \weq \int_{N(S)} e^{-\eta(u)} \, \Phi(\eta) P_{\mu} (d\eta).
\end{equation}

(v) \emph{Conclusion.}
Finally, we note that the formula $Q(d\eta) = \Phi(\eta) P_\mu(d\eta)$ defines a measure on $(N(S), \cN(S))$.
By applying \eqref{eq:PoissonDensityLaplace} with $u=0$, we see that 
$Q(N(S)) = \int_{N(S)} \Phi(\eta) P_{\mu} (d\eta) = P_\la(N(S)) = 1$, so that $Q$ is a probability measure.
Because the Laplace functional uniquely characterises \cite[Proposition 2.10]{Last_Penrose_2018} a probability measure on $(N(S), \cN(S))$, we conclude from \eqref{eq:PoissonDensityLaplace} that $Q = P_\la$. In other words, $\Phi$ is a density of $P_\la$ with respect to $P_\mu$.  As an element of $L_1(N(S), \cN(S), P_\mu)$, we see that $\Phi(\eta) = 1_M(\eta) e^{\ell(\eta)}$, because $P_\mu(\Omega)=1$.
\end{proof}

\subsection{Proof of Theorem~\ref{the:PPPRenyi}}
\label{sec:PoissonRenyiProof}

First, Lemma~\ref{the:PoissonRenyiFinite} proves the claim under an additional condition that
$\la$ and $\mu$ are finite measures on $S$.
This proof is different from the usual topological approach that is based on approximating the measurable sets of
$S$ by a finite sigma-algebra \cite{Liese_1975,Karr_1983,Karr_1991}, which usually requires $S$ to be a separable metric space.
Instead, the following proof is based on (i) representing a Poisson PP distribution with a finite intensity measure using
a Poisson-distributed number of IID random variables (see \cite{Kingman_1967,Reiss_1993,Last_Penrose_2018});
and (ii) representing a Poisson PP distribution with a sigma-finite intensity measure using a decomposition
with respect to a countable partition.


\begin{lemma}
\label{the:PoissonACFinite}
$P_\la \ll P_\mu$ $\implies$ $\la \ll \mu$
for any Poisson PP distributions with sigma-finite intensity measures.
Furthermore, the converse implication $\la \ll \mu$ $\implies$ $P_\la \ll P_\mu$
holds when the intensity measures are finite.
\end{lemma}
\begin{proof}
Assume that $P_\la \ll P_\mu$.  Consider a set $B \subset S$ such that $\mu(B) = 0$.
Define $C = \{\eta\colon \eta(B)>0\}$.
Recall that $\eta(B)$ is Poisson distributed with mean $\mu(B)$
when $\eta$ is sampled from $P_\mu$.  
Therefore, $P_\mu(C) = 1 - e^{-\mu(B)} = 0$.  Now $P_\la \ll P_\mu$
implies that $0 = P_\la(C) = 1 - e^{-\la(B)}$, from which we conclude that $\la(B) = 0$.
Hence $\la \ll \mu$.

For finite intensity measures $\la,\mu$, Theorem~\ref{the:PoissonDensityFinite} confirms the
reverse implication
$\la \ll \mu$ $\implies$ $P_\la \ll P_\mu$.
\end{proof}

\begin{lemma}
\label{the:PoissonRenyiFinite}
$\Ren_\al( P_\la \| P_\mu ) = \Tsa_\alpha(\la \| \mu)$ for all $\al \in \R_+$ and
all Poisson PP distributions $P_\la,P_\mu$ with finite intensity measures $\la,\mu$.
\end{lemma}
\begin{proof}
Let $P_\nu$ be a Poisson PP distribution with intensity measure $\nu = \la+\mu$.
Let $f = \frac{d\la}{d\nu}, g = \frac{d\mu}{d\nu}$ be densities of $\la,\mu$ with respect to $\nu$.
Such functions exist by the Radon--Nikodym theorem \cite[Theorem 2.10]{Kallenberg_2002}.
Theorem~\ref{the:PoissonDensityFinite}
implies that
$P_\la, P_\mu$ are absolutely continuous with respect to $P_\nu$, admitting likelihood ratios
$F = \frac{dP_\la}{dP_\nu}$ and $G = \frac{dP_\mu}{dP_\nu}$ given by
\begin{equation}
 \label{eq:PoissonDensities}
 \begin{aligned}
 F(\eta)
 &\weq 1_{M_f}(\eta) \, e^{\nu(1-f) + \eta(\log f)}, \\
 G(\eta)
 &\weq 1_{M_g}(\eta) \, e^{\nu(1-g) + \eta(\log g)}, \\
 \end{aligned}
\end{equation}
where $M_f = \{\eta \in N(S) \colon \eta\{f=0\} = 0, \, \eta(S) < \infty\}$
and $M_g$ is defined similarly, and we abbreviate $\nu(f) = \int f \, d\nu$.

Let us first consider the case with $\al \notin \{0,1\}$.
Then $\Ren_\al(P_\la \| P_\mu) = \frac{1}{\al-1} \log Z_\al$
where
$Z_\al = \int_{N(S)} F^\al G^{1-\al} \, dP_\nu$.
By the standard conventions $0 \cdot \infty = \frac{0}{0} = 0$ and $\frac{1}{0} = \infty$,
we see that
\begin{equation}
 \label{eq:PoissonZAlpha}
 Z_\al
 \weq
 \begin{cases}
  \tilde Z_\al, &\quad \al \in (0,1),\\
  \tilde Z_\al + \infty \cdot P_\nu\{F>0, G=0\}, &\quad \al \in (1,\infty),
 \end{cases}
\end{equation}
where
\[
 \tilde Z_\al
 \weq \int_{F>0, \, G>0} F^\al G^{1-\al} \, dP_\nu.
\]
To derive a simplified expression for $\tilde Z_\al$, define
$U = \{f > 0, \, g > 0\}$ and consider a set of point patterns
$N(U) = \{\eta \in N(S) \colon \eta(S) < \infty, \, \eta(U^c) = 0\}$.
In light of \eqref{eq:PoissonDensities}, we see that
$\{F>0, \, G>0\} \cap \{\eta \colon \eta(S) < \infty\} = M_f \cap M_g = N(U)$.
Because $P_\nu\{\eta \colon \eta(S) < \infty\} = 1$, it follows that 
\[
 \tilde Z_\al
 \weq \int_{N(U)} F^\al G^{1-\al} \, dP_\nu.
\]
We also note that for $\eta \in N(U)$,
\begin{equation}
\label{eq:HellingerIntegrand1}
 \begin{aligned}
 F(\eta)^\al G(\eta)^{1-\al}
 &\weq e^{\al \nu(1-f) + (1-\al)\nu(1-g)} e^{\eta(\log h_\al)},
 \end{aligned}
\end{equation}
where $h_\al = f^\al g^{1-\al}$.
The conditional distribution of a point pattern $\eta$ sampled from $P_\nu$ given $\eta \in N(U)$ equals (Proposition~\ref{the:PoissonTruncation}) $P_{\nu_U}$.   By also noting that $P_\nu(N(U)) = e^{-\nu(U^c)}$,
it follows that
\[
 \int_{N(U)} e^{\eta(\log h_\al)} \, P_\nu(d\eta)
 \weq e^{-\nu(U^c)} \int_{N(S)} e^{\eta(\log h_\al)} \, P_{\nu_U}(d\eta).
\]
Because $\int_S ( \abs{ \log h_\al} \wedge 1) \, d\nu_U \le \nu(U)  < \infty$ and $x \mapsto \log h(x)$ restricted to $U$ is $\R$-valued, the Laplace functional formula of Poisson point patterns \cite[Lemma 12.2]{Kallenberg_2002} implies that
$\int_{N(S)} e^{\eta(\log h_\al)} \, P_{\nu_U}(d\eta) = e^{\nu_U(h_\al - 1)}$.
By integrating \eqref{eq:HellingerIntegrand1} with respect to $P_\nu$, it follows that
\begin{equation}
 \label{eq:PoissonTildeZAlpha}
 \begin{aligned}
  \tilde Z_\al
  &\weq e^{\al \nu(1-f) + (1-\al)\nu(1-g)} e^{-\nu(U^c)} e^{\nu_U(h_\al - 1)} \\
  &\weq e^{-\al \nu(f) - (1-\al)\nu(g) + \nu_U(h_\al)}.
 \end{aligned}
\end{equation}

We are now ready to verify the claim by considering the following five cases one by one:
\begin{enumerate}[(i)]
\item
Assume now that $\al \in (0,1)$.
Then by \eqref{eq:PoissonZAlpha}, we see that
$\Ren_\al( P_\la \| P_\mu ) = \frac{1}{\al-1} \log \tilde Z_\al$.
We also note that $\nu_U(h_\al) = \nu(h_\al)$, so that 
by \eqref{eq:PoissonTildeZAlpha}, we conclude that
\begin{equation}
 \label{eq:PoissonRenyiTsallis1}
 \Ren_\al( P_\la \| P_\mu )
 \weq \frac{\nu(\al f + (1-\al) g - h_\al ) }{1-\al}.
\end{equation}
The claim follows because the right side equals $\Tsa_\al(\la \| \mu)$ by \eqref{eq:Tsallis}.

\item Assume now that $\al \in (1,\infty)$ and $\la \ll \mu$.
Then $P_\la \ll P_\mu$ by Lemma~\ref{the:PoissonACFinite}.
Then $\nu\{f>0, g=0\} = 0$ and $P_\nu\{F>0, G=0\} = 0$ (Lemma~\ref{the:Density}).
In this case we again find that $\nu_U(h_\al) = \nu(h_\al)$.
Therefore, in light of \eqref{eq:PoissonZAlpha} and \eqref{eq:PoissonTildeZAlpha},
we see that \eqref{eq:PoissonRenyiTsallis1} holds also in this case, and the claim follows.

\item Assume now that $\al \in (1,\infty)$ and $\la \not\ll \mu$.
Then $P_\la \not\ll P_\mu$ by Lemma~\ref{the:PoissonACFinite}.
Then $\nu\{f>0, g=0\} > 0$ and $P_\nu\{F>0, G=0\} > 0$ (Lemma~\ref{the:Density}).
Hence by \eqref{eq:PoissonZAlpha}, $\Ren_\al(P_\la \| P_\mu) = \infty$.
The assumption $\al > 1$ implies that $f^\al g^{1-\al} = \infty$ on the set $\{f>0, \, g=0\}$.
Therefore $\int_S f^\al g^{1-\al} \, d\nu = \infty$, and we find that $\Tsa_\al(\la \| \mu) = \infty$
by \eqref{eq:Tsallis}, and the claim follows.

\item Assume that $\al = 1$. Now $\Ren_1(P_\la \| P_\mu) = \lim_{\al \uparrow 1} \Ren_\al(P_\la \| P_\mu)$
\cite{VanErven_Harremoes_2014}. By (i), we know that
$\Ren_\al( \la \| \mu) = \Tsa_\al( \la \| \mu)$ for all $\al \in (0,1)$.
The claim follows by letting $\alpha \uparrow 1$ and noting that
$\Tsa_1(P_\la \| P_\mu) = \lim_{\al \uparrow 1} \Tsa_\al(P_\la \| P_\mu)$
by Theorem~\ref{the:Tsallis}.

\item Assume that $\al = 0$. 
Formula \eqref{eq:Renyi} shows that
$\Ren_0( P_\la \| P_\mu ) = - \log P_\mu(F>0) = - \log P_\mu( M_f )$.
Because $P_\mu( M_f ) = e^{-\mu\{f=0\}}$,
it follows that $\Ren_0( P_\la \| P_\mu ) = \mu\{f=0\}$.
By formula \eqref{eq:Tsallis}, we see that $\Ren_0( P_\la \| P_\mu ) = \Tsa_0( \la \| \mu )$.
\end{enumerate}
\end{proof}


With the help of Lemma~\ref{the:PoissonRenyiFinite} we will prove Theorem~\ref{the:PPPRenyi} in the general case where $\la,\mu$ are sigma-finite measures on $S$.

\begin{proof}[Proof of Theorem~\ref{the:PPPRenyi}]
(i) Fix $\al \in (0,\infty)$.
Define $\nu = \la+\mu$. Then $\nu$ is sigma-finite.
Select a partition (see Lemma~\ref{the:SigmaFinitePartition}) $S = \cup_{n \ge 1} S_n$ such that $\nu(S_n) < \infty$ for all $n$.
Denote $N(S_n) = \{\eta \in N(S) \colon \eta(S_n^c) = 0\}$.
Define $\tau \colon N(S) \to \prod_{n=1}^\infty N(S_n)$ by
\begin{equation}
 \label{eq:PartitionProjection}
 \tau(\eta)
 \weq (\eta_{S_1}, \eta_{S_2}, \dots),
\end{equation}
where $\eta_{S_n} \in N(S_n)$ is defined by $\eta_{S_n}(B) = \eta(B \cap S_n)$.
A restriction theorem \cite[Theorem 5.2]{Last_Penrose_2018} implies that when $\eta$ is sampled from $P_\la$, then the restrictions $\eta_{S_1}, \eta_{S_2}, \dots$ are mutually independent Poisson PPs with intensity measures $\la_{S_1}, \la_{S_2}, \dots$ defined by $\la_{S_n}(B) = \la(B \cap S_n)$.  Therefore, the pushforward probability measure $P_\la \circ \tau^{-1}$ can be written as a product of Poisson PP distributions $P_{\la_{S_n}}$, $n \ge 1$.  The same reasoning is valid also for $P_\mu$.
Hence
\[
 P_\la \circ \tau^{-1} \weq \bigotimes_{n=1}^\infty P_{\la_{S_n}}
 \qquad \text{and} \qquad
 P_\mu \circ \tau^{-1} \weq \bigotimes_{n=1}^\infty P_{\mu_{S_n}}.
\]
Because the sets $S_1,S_2, \dots$ form a partition of $S$, we see that the map $\tau$ defined by \eqref{eq:PartitionProjection} is a bijection with inverse $\tau^{-1}(\eta_1,\eta_2,\dots) = \sum_{n \ge 1} \eta_n$.  Standard arguments show that $\tau$ and $\tau^{-1}$ are measurable mappings (see Section~\ref{sec:PartitionDecompositionMeasurable}).
Lemma~\ref{the:RenyiBijection} then implies that 
$\Ren_\al( P_\la \| P_\mu ) = \Ren_\al( P_\la \circ \tau^{-1} \| P_\mu \circ \tau^{-1} )$.
Because \Renyi divergences of order $\al > 0$ factorise over tensor products \cite[Theorem 28]{VanErven_Harremoes_2014} it follows that
\begin{equation}
 \label{eq:RenyiPoissonPartition}
 \Ren_\al( P_\la \| P_\mu )
 \weq \sum_{n=1}^\infty \Ren_\al( P_{\la_{S_n}} \| P_{\mu_{S_n}} ).
\end{equation}

Because $\la,\mu \ll \nu$ and $\nu$ is sigma-finite, the Radon--Nikodym theorem \cite[Theorem 2.10]{Kallenberg_2002} implies that there exist densities
$f = \frac{d\la}{d\nu}$ and $g = \frac{d\mu}{d\nu}$ of $\la$ and $\mu$ with respect to $\nu$.
Observe now that $\la_{S_n}(A) = \int_{A \cap S_n} f \, d\nu
= \int_{A} f \, d\nu_{S_n}$ for all measurable $A \subset S$.
Similarly, $\mu_{S_n}(A) = \int_{A} g \, d\nu_{S_n}$.
We conclude that the functions $f$ and $g$ also act as densities 
$f = \frac{d\la_{S_n}}{d\nu_{S_n}}$ and $g = \frac{d\mu_{S_n}}{d\nu_{S_n}}$
of the finite measures $\la_{S_n}, \mu_{S_n}$ with respect to $\nu_{S_n}$.
Lemma~\ref{the:PoissonRenyiFinite} now implies that
\begin{equation}
 \label{eq:RenyiPoissonTruncated}
 \Ren_\al( P_{\la_{S_n}} \| P_{\mu_{S_n}} ) = \Tsa_\al(\la_{S_n} \| \mu_{S_n}).
\end{equation}
Furthermore, by \eqref{eq:TsallisPoissonNew} in Theorem~\ref{the:Tsallis}, we see that 
\[
 \Tsa_\al(\la_{S_n} \| \mu_{S_n})
 \weq \int_S \Ren_\al(p_{f(x)} \| p_{g(x)}) \, \nu_{S_n}(dx),
\]
where $p_s$ refers to the Poisson distribution $k \mapsto e^{-s k} \frac{s^k}{k!}$ with mean $s$.  Observe that the integrand on the right side above is nonnegative, $d \nu_{S_n} = 1_{S_n} d\nu$, and $\sum_n 1_{S_n} = 1$.  Fubini's theorem combined with \eqref{eq:TsallisPoissonNew} then implies that
\begin{align*}
 \sum_{n=1}^\infty \Tsa_\al(\la_{S_n} \| \mu_{S_n})
 \weq \int_S \Ren_\al(p_{f(x)} \| p_{g(x)}) \, \nu(dx)
 \weq \Tsa_\al(\la \| \mu).
\end{align*}
By combining this with \eqref{eq:RenyiPoissonPartition} and \eqref{eq:RenyiPoissonTruncated}, 
it follows that $\Ren_\al( P_\la \| P_\mu ) = \Tsa_\al(\la \| \mu)$.

(ii) Finally, let use verify that $\Ren_0( P_\la \| P_\mu ) = \Tsa_0(\la \| \mu)$ under the additional
assumption that $\Tsa_\beta(\la \| \mu) < \infty$ for some $\beta > 0$.
We saw in part (i) of the proof that
\begin{equation}
 \label{eq:PoissonRenyiTsallisNonero}
 \Ren_{\al}(\la \| \mu) = \Tsa_{\al}(\la \| \mu)
 \quad \text{for all $\al \in (0,\infty)$}.
\end{equation}
Now \cite[Theorem 7]{VanErven_Harremoes_2014} implies that $\al \mapsto \Ren_{\al}(P_\la \| P_\mu)$ is continuous
on $[0,1]$, and Theorem~\ref{the:Tsallis} implies that $\al \mapsto \Tsa_{\al}(\la \| \mu)$
is continuous on $[0,\beta]$.
Hence the claim follows by taking limits $\al \to 0$ in \eqref{eq:PoissonRenyiTsallisNonero}.
\end{proof}

\section{Conclusions}
\label{sec:Conclusions}

By developing an analytical toolbox of generalised Tsallis divergences for sigma-finite measures,
a framework was derived for analysing likelihood ratios and \Renyi divergences of
Poisson PPs on general measurable spaces.  The main advantage
of this approach is that it is purely information-theoretic and free of
topological assumptions.  
This framework allows one to derive explicit descriptions of
Kullback--Leibler divergences, \Renyi divergences, Hellinger distances, and likelihood ratios
for statistical models that admit a measurable one-to-one map into a space of point patterns
governed by a Poisson PP distribution.
Marked Poisson PPs corresponding to Poisson PPs on abstract product spaces provide a rich
context for various applications.  The disintegrated Tsallis divergence formula in Section~\ref{sec:TsallisKernel} is key
to understanding their information-theoretic features.   For completing the general theory,
understanding whether the technical condition \eqref{eq:TsallisKernelMeasurable} is necessary in
Theorem~\ref{the:TsallisKernel} remains an important open problem.

Future directions of extending this work include deriving similar results for a wider class of
statistical models derived from Poisson processes and marked point patterns, for example Cox processes, Hawkes processes,
Poisson shot noise random measures, and \Matern point patterns.
A challenge here is that mechanisms used to derive the model from a marked point pattern tend to lose information.
For example, Poisson shot noise models in certain limiting regimes reduce to Gaussian white noises \cite{Kaj_Leskela_Norros_Schmidt_2007}.
The thorough characterisation of Poisson PP distributions derived in this article
is expected to serve as a cornerstone for this type of further studies.

\section*{Acknowledgments}

The author thanks Venkat Anantharam for insightful discussions and 
two anonymous reviewers for valuable remarks that have helped to improve the presentation.

\appendix

\section{Measure theory}

\subsection{Densities}

Recall notations from Section~\ref{sec:Measures}.

\begin{lemma}
\label{the:Density}
Let $\la, \mu$ be measures on a measurable space $(S, \cS)$ admitting densities $f = \frac{d\la}{d\nu}$ and $g = \frac{d\mu}{d\nu}$ with respect to a measure $\nu$. 
\begin{enumerate}[(i)]
\item $\la\{f=0\} = 0$ and $\mu\{g=0\}=0$.
\item For any measurable set $B$, $\la(B) = 0$ if and only if $\nu( \{f>0\} \cap B ) = 0$.
\item $\la \ll \mu$ if and only if $\la\{g=0\} = 0$.
\item $\la \ll \mu$ if and only if $\nu\{f >0, \, g = 0\} = 0$.
\item $\la \perp \mu$ if and only if $\nu\{f >0, \, g > 0\} = 0$.
\item $\la \perp \mu$ if and only if $\la\{g > 0\} = 0$.
\end{enumerate}
\end{lemma}
\begin{proof}
(i) $\la\{f = 0\}
= \int_{\{f=0\}} f \, d\nu
= 0.
$
Analogously we see that $\mu\{g=0\}=0$.

(ii) Fix a measurable set $B$, and denote $A = \{f > 0\} \cap B$.
Note that (i) implies that $\la(B) = \la(A) = \int_A f \, d\nu$.
Hence $\nu(A) = 0$ implies that $\la(B) = 0$.
Assume next that $\nu(A) > 0$.  Then $\nu(A_n) > 0$ for some integer $n \ge 1$,
where $A_n = \{f \ge n^{-1}\} \cap B$.
Hence $\la(B) \ge \la(A_n) = \int_{A_n} f \, d\nu \ge n^{-1} \nu(A_n) > 0$.

(iii) Assume that $\la\{g=0\} = 0$, and consider a set $B$ such that $\mu(B) = 0$. By applying (ii) for $\mu$, we find that
$\nu(\{g>0\} \cap B) = 0$.  Then $\la \ll \nu$ implies that $\la(\{g>0\} \cap B) = 0$.
Then $\la(B) = \la(\{g>0\} \cap B) + \la(\{g=0\} \cap B) = 0$ due to assumption $\la\{g=0\} = 0$.
Hence $\la(B) = 0$, and we conclude that $\la \ll \mu$.
The converse implication $\la \ll \mu \implies \la\{g=0\} = 0$ is immediate from (i).

(iv) By applying (ii) with $B = \{g=0\}$,  we find that $\la\{g=0\}$ is equivalent to $\nu\{f>0, \, g=0\} = 0$.
The claim now follows by (iii).

(v) Assume that $\nu\{f >0, \, g > 0\} = 0$.  Let $B = \{f>0\}$. Then $\la(B^c) = 0$ due to (i).
Furthermore by (i), $\mu(B) = \mu(B \cap \{g > 0\}) = \mu\{f >0, \, g > 0\}$. Hence $\mu(B=0)$ due to $\mu \ll \nu$.
Hence $\la \perp \mu$.  Assume now that $\la \perp \mu$.
Then there exists a set $B$ such that $\la(B^c) = 0$ and $\mu(B)=0$.
Then (ii) implies that $\nu( \{f>0\} \cap B^c ) = 0$ and $\nu( \{g>0\} \cap B ) = 0$. Then
\begin{align*}
 \nu( \{f>0, g > 0\} )
 \wle \nu( \{f>0\} \cap B^c ) + \nu( \{g > 0\} \cap B )
 \weq 0.
\end{align*}

(vi) Assume that $\la\{g > 0\} = 0$. Let $B = \{g=0\}$.
Then $\la(B^c) = 0$, and $\mu(B)=0$ due to (i).
Hence $\la \perp \mu$.  Assume now that $\la \perp \mu$.
Then there exists a set $B$ such that $\la(B^c) = 0$ and $\mu(B)=0$. Then by (ii), we see that
$\nu( \{g>0\} \cap B ) = 0$. Then $\la \ll \nu$ implies that $\la( \{g>0\} \cap B ) = 0$.
Then
\begin{align*}
 \la\{g > 0\}
 \weq \la( \{g > 0\} \cap B ) + \la( \{g > 0\} \cap B^c )
 \weq 0.
\end{align*}
\end{proof}

\subsection{Basic measure theory}

\begin{lemma}
\label{the:SigmaFinitePartition}
If $\nu$ is a sigma-finite measure on a measurable space $(S,\cS)$,
then there exists a partition $S = \cup_{n \ge 1} S_n$ such that $\nu(S_n) < \infty$ for all $n$.
\end{lemma}
\begin{proof}
Because $\nu$ is sigma-finite, there exist measurable sets such that $\cup_{n \ge 1} C_n = S$ and
$\nu(C_n) < \infty$ for all $n$.
Define $S_0 = \emptyset$ and $S_n = C_{n} \setminus S_{n-1}$ for $n \ge 1$.
Then the sets $S_1,S_2,\dots$ are mutually disjoint, and $S = \cup_{n \ge 1} S_n$, together with $\nu(S_n) \le \nu(C_n) < \infty$ for all $n$.
\end{proof}

The following result is a convenient alternative form of Lebesgue's dominated convergence theorem
that quantifies uniform integrability (boundedness in the increasing convex stochastic order \cite{Leskela_Vihola_2013})
in a flexible manner.

\begin{lemma}
\label{the:DOM}
Let $f, f_1, f_2,\dots$ and $g_1,g_2,\dots$ be measurable real-valued functions on $(S,\cS)$
such that $f_n \to f$, $g_n \to g$, $\abs{f_n} \le g_n$ for all $n$, and $\sup_n \int_S g_n \, d\mu < \infty$.
Then $\int_S f_n \, d\mu \to \int_S f \, d\mu$,
and $\int_S \abs{f} \, d\mu \le \sup_n \int_S g_n \, d\mu$.
\end{lemma}
\begin{proof}
We note that $g = \liminf_{n \to \infty} g_n$, and that Fatou's lemma \cite[Lemma 1.20]{Kallenberg_2002} implies
$\int g \, d\mu = \int ( \liminf_{n \to \infty} g_n ) \, d\mu \le \liminf_{n \to \infty} \int g_n \, d\mu
\le \sup_n \int g_n \, d\mu$. Then $\int g \, d\mu$ is finite, and Kallenberg's version of
Lebesgue's dominated convergence theorem \cite[Theorem 1.21]{Kallenberg_2002}
yields the first claim.  For the second claim, we note by Fatou's lemma that
$\int_S \abs{f} \, d\mu
= \int_S \liminf \abs{f_n} \, d\mu
\le \liminf \int_S \abs{f_n} \, d\mu
\le \sup_n \int_S g_n \, d\mu$.
\end{proof}

\subsection{Measurable bijections}

\begin{lemma}
\label{the:RenyiBijection}
Let $S,T$ be measurable spaces, and let $\phi \colon S \to T$ be a measurable bijection with a measurable inverse.
Then $\Ren_\al( P \circ \phi^{-1} \| \, Q \circ \phi^{-1} ) = \Ren_\al( P \| Q )$ for all $\al > 0$.
\end{lemma}
\begin{proof}
Let $P,Q$ be probability measures on $S$. Fix densities $p = \frac{dP}{dm}$ and $q = \frac{dQ}{dm}$ with respect to $m=P+Q$.
Define a function $\tilde p \colon T \to \R_+$ by $\tilde p = p \circ \phi^{-1}$ and a measure $\tilde m$
on $T$ by $\tilde m = m \circ \phi^{-1}$.  Note that for any measurable $A \subset T$,
\begin{align*}
 \int_A \tilde p \, d \tilde m
 \weq \int_{T} 1_A(y) \tilde p(y) \, \tilde m(dy)
 \weq \int_S 1_A(\phi(x)) \tilde p(\phi(x)) \, m(dx).
\end{align*}
Because $\tilde p(\phi(x)) = p(x)$ for all $x$, we see that
\begin{align*}
 \int_A \tilde p \, d \tilde m
 \weq \int_{\phi^{-1}(A)} p(x) \, m(dx)
 \weq P( \phi^{-1}(A)).
\end{align*}
We conclude that $\tilde p = \frac{d P \circ \phi^{-1}}{d m \circ \phi^{-1}}$ is a density of $P \circ \phi^{-1}$ with respect to $\tilde m$.
Similarly, we see that $\tilde q = q \circ \phi^{-1}$ is a density of
$Q \circ \phi^{-1}$ with respect to $\tilde m$. Hence,
\begin{align*}
 \int_T (p \circ \phi^{-1})^\al (q \circ \phi^{-1})^{1-\al} d \tilde m
 &\weq \int_S (p \circ \phi^{-1}(\phi(x)))^\al (q \circ \phi^{-1}(\phi(x))^{1-\al} m(dx) \\
 &\weq \int_S (p(x))^\al (q(x))^{1-\al} m(dx) \\
 &\weq \int_S p^\al q^{1-\al} \, dm.
\end{align*}
From this the claim follows for $\al \ne 1$. The case with $\al=1$ is similar.
\end{proof}

\section{Point patterns}

\subsection{Measurability of sigma-finite decompositions}
\label{sec:PartitionDecompositionMeasurable}

This section discusses a decomposition of a point pattern with respect to a countable partition of the ground space $S$.
Let $(S,\cS)$ be a measurable space.
Let $N(S)$ be the set of point patterns (measures with values in $\Z_+ \cup \{\infty\}$)
on $(S, \cS)$ equipped with the sigma-algebra
$\cN(S) = \sigma(\ev_B \colon B \in \cS)$ generated by the evaluation maps $\ev_B \colon \eta \mapsto \eta(B)$.

Assume that $S_1,S_2, \dots \in \cS$ are disjoint and such that
$S = \cup_{n=1}^\infty S_n$.
We define $N(S_n) = \{\eta \in N(S) \colon \eta(S_n^c) = 0\}$
and equip this set with
the trace sigma-algebra $\cN(S_n) = \cN(S) \cap N(S_n)$.
We define the truncation map $\tau_{S_n} \colon N(S) \to N(S_n)$ by
\[
 (\tau_{S_n} \eta) (B) \weq \eta(B \cap S_n), \quad B \in \cS.
\]
Then we define $\tau \colon N(S) \to \prod_{n=1}^\infty N(S_n)$ by
\begin{equation}
 \label{eq:PartitionProjectionTemp}
 \tau(\eta)
 \weq (\tau_{S_1}(\eta), \tau_{S_2}(\eta), \dots).
\end{equation}
We find that $\tau$ is a bijection with 
inverse
\[
 \tau^{-1}( \eta_1, \eta_2, \dots) 
 \weq \sum_{n=1}^\infty \eta_n.
\]
We equip $\prod_{n=1}^\infty N(S_n)$ with the product sigma-algebra $\bigotimes_{n=1}^\infty \cN(S_n)$.

\begin{lemma}
\label{the:PartitionDecompositionMeasurable}
$\tau \colon (N(S), \cN(S)) \to (\prod_{n=1}^\infty N(S_n), \bigotimes_{n=1}^\infty \cN(S_n))$ is a measurable bijection with a measurable inverse.
\end{lemma}
\begin{proof}
Denote $\cC(S) = \{ \ev_B^{-1}(\{k\}) \colon B \in \cS, \, k \in \Z_+\}$ and note that
this set family generates the sigma-algebra $\cN(S)$.
By \cite[Lemma II.3.3]{Shiryaev_1996}, we know that the
set family $\cC(S) \cap N(S_n)$ generates the trace sigma-algebra $\cN(S_n) = \cN(S) \cap N(S_n)$.

We start by by verifying that $\tau_{S_n} \colon N(S) \to N(S_n)$ is measurable.
Fix a set $B \in \cS$ and an integer $k \ge 0$, and consider
a set in $\cC(S) \cap N(S_n)$ of form
\[
 C
 \weq \{\eta \in N(S) \colon \eta(S_n^c) = 0, \, \eta(B) = k\}.
\]
Then
\begin{align*}
 \tau_{S_n}^{-1}(C)
 &\weq \{\eta \in N(S) \colon \eta(S_n^c \cap S_n) = 0, \, \eta(B \cap S_n) = k\} \\
 &\weq \{\eta \in N(S) \colon \eta(B \cap S_n) = k\}
\end{align*}
shows that $\tau_{S_n}^{-1}(C) \in \cN(S)$.
Because such sets $C$ generate $\cN(S_n)$,
it follows \cite[Lemma 1.4]{Kallenberg_2002} that $\tau_{S_n}$ is measurable.
Because each coordinate map of $\tau$ is measurable, it follows 
\cite[Lemma 1.8]{Kallenberg_2002} that $\tau$ is measurable.

Let us now verify that the inverse map
$\tau^{-1} \colon \prod_{n=1}^\infty N(S_n) \to N(S)$
is measurable.
Fix a set $B \in \cS$ and an integer $k \ge 0$, and consider a set in $\cC(S)$ of form
\[
 C
 \weq \{\eta \in N(S) \colon \eta(B) = k\}.
\]
Then
\begin{align*}
 (\tau^{-1})^{-1}(C)
 &\weq \{(\eta_1,\eta_2,\dots) \colon \sum_{n=1}^\infty \eta_n(B) = k\}.
\end{align*}
Let $Z_k$ be the collection of integer-valued measures $z = \sum_{n=1}^\infty z_n \delta_n$ on $\N = \{1,2,\dots\}$
with total mass $\sum_{n=1}^\infty z_n = k$. Then
\begin{align*}
 (\tau^{-1})^{-1}(C)
 &\weq \bigcup_{z \in Z_k} \{(\eta_1,\eta_2,\dots) \colon \eta_n(B) = z_n \ \text{for all $n$}\} \\
 &\weq \bigcup_{z \in Z_k} \bigcap_{n=1}^\infty \{(\eta_1,\eta_2,\dots) \colon \eta_n(B) = z_n \}
\end{align*}
shows that $(\tau^{-1})^{-1}(C) \in \bigotimes_{n=1}^\infty \cN(S_n)$.
Because such sets $C$ generate $\cN(S)$,
it follows \cite[Lemma 1.4]{Kallenberg_2002} that $\tau^{-1}$ is measurable.
\end{proof}

\subsection{Compensated Poisson integrals}
\label{sec:CompensatedPoissonIntegral}

Given measurable sets $S_n \uparrow S$ and measures $\eta, \mu$ on a measurable space $(S,\cS)$, we
say that the \new{compensated integral}
\begin{equation}
 \label{eq:CompensatedIntegral}
 \int_A f \, d(\eta-\mu)
 \weq \lim_{n \to \infty} \left( \int_{A \cap S_n} f \, d\eta - \int_{A \cap S_n} f \, d\mu \right)
\end{equation}
of a measurable function $f \colon S \to \R$ over a
measurable set $A \subset S$ converges if 
$\int_{A \cap S_n} \abs{f} \, d\eta + \int_{A \cap S_n} \abs{f} \, d\mu < \infty$ for all $n$,
and the limit in \eqref{eq:CompensatedIntegral} exists in $\R$.
The following two results characterise the convergence of
compensated integrals when $\eta$ is sampled from a Poisson PP distribution
with intensity measure $\mu$.  These are needed for proving Theorem~\ref{the:PoissonDensity}.

\begin{lemma}
\label{the:CompensatedPoissonIntegral}
Let $P_\mu$ be a Poisson PP distribution with a sigma-finite intensity measure $\mu$ such that $\mu(S_n) < \infty$ for all $n$.  For any bounded function $f \colon S \to \R$ such that $\int_A f^2 \, d\mu < \infty$,
the compensated integral $\int_A f \, d(\eta-\mu)$ converges for $P_\mu$-almost every $\eta \in N(S)$.
\end{lemma}
\begin{proof}
Define $U_n = S_n \setminus S_{n-1}$ for $n \ge 1$, where $S_0 = \emptyset$.
Because $f$ is bounded, we see that $\int_{U_n} \abs{f} \, d\mu \le \supnorm{f} \mu(S_n) < \infty$.
Campbell's theorem \cite[Section 3.2]{Kingman_1993} then implies that
\[
 W_n(\eta)
 \weq \int_{U_n} f \, d\eta - \int_{U_n} f \, d\mu
\]
defines a real-valued random variable on probability space $(N(S), \cN(S), P_\mu)$
with mean $E_\mu W_n = 0$ and variance $E_\mu W_n^2 = \int_{U_n} f^2 \, \mu$.
Because the sets $U_n$ are disjoint, the random variables $W_n$ are independent. 
Furthermore, $E_\mu \sum_{n=1}^\infty W_n^2 = \int_S f^2 \, d\mu < \infty$.
The Khinchin--Kolmogorov variance criterion \cite[Lemma 4.16]{Kallenberg_2002} then implies
that the sum $W = \sum_{n=1}^\infty W_n$ converges almost surely.
Hence $W$, or equivalently the right side of formula \eqref{eq:CompensatedIntegral},
is a well-defined real-valued random variable on the probability space $(N(S), \cN(S), P_\mu)$.
In particular $W(\eta) \in \R$ for $P_\mu$-almost every $\eta$.
\end{proof}

\begin{lemma}
\label{the:PoissonDensityIntegral}
Let $P_\mu$ be a Poisson PP distribution with a sigma-finite intensity measure $\mu$.
Let $\phi \colon S \to \R_+$ be such that $\int_S (\sqrt{\phi}-1)^2 \, d\mu < \infty$,
and denote $A = \{x \in S \colon \abs{\log\phi(x)} \le 1 \}$.
Assume that $S_n \uparrow S$ and $\mu(S_n) < \infty$ for all $n$.
Then the sets
\[
 \Omega_1
 \weq \left\{ \eta \in N(S) \colon \int_{A^c \cap \{\phi > 0\}} \nhquad \abs{\log \phi} \, d\eta < \infty, \ \eta(A^c) < \infty \right\}
\]
and
\[
 \Omega_2
 \weq \left\{ \eta \in N(S) \colon \text{$\int_A \log \phi \, d(\eta-\mu)$ converges}, \
 \eta(S_n) < \infty \ \text{for all $n$} \right\}
\]
satisfy $P_\mu(\Omega_1)=1$ and $P_\mu(\Omega_2)=1$.
\end{lemma}

\begin{proof}
We note that $\mu(A^c) = \int_{A^c} d\mu \le \int_{A^c} (\phi+1) \, d\mu$.
By \eqref{eq:SquareRootDiff}, $\phi+1 \wle \frac{e + 1}{(e^{1/2}-1)^2} (\sqrt{\phi}-1)^2$
on $A^c$.
It follows that $\mu(A^c)$ is finite.
Campbell's formula then implies $\int_{N(S)}\eta(A^c) P_\mu(d\eta) = \mu(A^c) < \infty$.
Hence the set $\Omega_1' = \{\eta \colon \eta(A^c) < \infty\}$ satisfies $P_\mu(\Omega_1') = 1$.
We also note that
for any $\eta \in \Omega_1'$, the restriction of $\eta$ into $A^c$ can be written
as a finite sum $\eta_{A^c} = \sum_i \delta_{x_i}$ with $x_i \in S$ such that $\phi(x_i) > 0$,
so that 
$\int_{A^c \cap \{\phi > 0\}} \abs{\log \phi} \, d\eta = \sum_i \abs{\log \phi(x_i)}$ is finite.
Therefore, $\Omega_1' = \Omega_1$, and we conclude that $P_\mu(\Omega_1) = 1$.

Let
\[
 \Omega_2'
 \weq \Big\{ \eta \in N(S) \colon \eta(S_n) < \infty \ \text{for all $n$} \Big\}.
\]
Because $\mu(S_n) < \infty$ for all $n$, we see that $P_\mu(\Omega_2')=1$.
Observe that $f = 1_A \log \phi$ is bounded, and
by \eqref{eq:A3},
\[
 \int_S f^2 \, d\mu
 \weq \int_A \log^2\phi \, d\mu
 \wle 4 e^3 \int_S (\sqrt{\phi}-1)^2
 \ < \ \infty.
\]
Lemma~\ref{the:CompensatedPoissonIntegral}
implies that $P_\mu(\Omega_2)=1$. 
\end{proof}

\subsection{Truncated Poisson PPs}
\label{sec:PoissonTruncation}

Let $(S,\cS)$ be a measurable space.  Given a set $U \in \cS$ and a measure $\la$,
define $\la_U(A) = \la(A \cap U)$ for $A \in \cS$.  Then $\la_U$ is a measure on $(S,\cS)$.
We denote the truncation map by $\pi_U \colon \la \mapsto \la_U$.   

When $\eta$ is sampled from a Poisson PP distribution $P_\la$ with a sigma-finite intensity measure $\la$ on $S$, then the probability distribution of the random point pattern $\eta_U$ is given by $P_\la \circ \pi_U^{-1}$. The following proposition confirms that the law of $\eta_U$ is a Poisson PP distribution with truncated intensity measure $\lambda_U$, and that that $P_{\la_U}$ is also the conditional distribution of $\eta$ sampled from $P_\la$ given that $\eta(U^c) = 0$.

\begin{proposition}
\label{the:PoissonTruncation}
For any measurable set $U \subset S$:
\begin{enumerate}[(i)]
\item $P_\la \circ \pi_U^{-1} = P_{\la_U}$.
\item $P_\la\{ \eta \colon \eta \in A, \, \eta(U^c) = 0\} = e^{-\la(U^c)} P_{\la_U}(A)$ for all measurable $A \subset S$.
\end{enumerate}
\end{proposition}
\begin{proof}
(i) This follows by \cite[Theorem 5.2]{Last_Penrose_2018}.

(ii) We note that $\eta = \eta_U$ for any $\eta$ such that $\eta(U^c) = 0$,
and that $\eta(U^c) = 0$ if and only if $\eta_{U^c} = 0$. We also note (\cite[Theorem 5.2]{Last_Penrose_2018}) that
the random point patterns $\eta_U$ and $\eta_{U^c}$ are independent when $\eta$
is sampled from $P_\la$. Therefore, by (i),
\begin{align*}
 P_\la\{ \eta \colon \eta \in A, \, \eta(U^c) = 0\}
 &\weq P_\la\{ \eta \colon \eta_U \in A, \, \eta_{U^c} = 0\} \\
 &\weq P_{\la_U}(A) \, P_{\la_{U^c}}( \{0\} ).
\end{align*}
By noting that $P_{\la_{U^c}}( \{0\} ) = P_{\la_{U^c}}( \eta(S) = 0 )$ and noting that
$\eta(S)$ is Poisson-distributed with parameter $\la(U^c)$ when $\eta$ is sampled from
$P_{\la_{U^c}}$, we see that $P_{\la_{U^c}}( \{0\} ) = e^{-\la(U^c)}$.
Hence the claim follows.
\end{proof}

\section{Elementary analysis}

\begin{lemma}
\label{the:Log}
For all $x > -1$,
\begin{align}
 \label{eq:Log1a}
 \frac{x}{1+x} &\wle \log(1+x) \wle x, \\
 \label{eq:Log1}
 x - \frac{x^2}{2 (1-x_-)^2} &\wle \log(1+x) \wle x - \frac{x^2}{2 (1+x_+)^2},
\end{align}
and for all $y > 0$,
\begin{equation}
 \label{eq:Log2a}
 \abs{\log y} \wle \frac{\abs{y-1}}{y \wedge 1}
\end{equation}
and
\begin{equation}
 \label{eq:Log2}
 0 \wle 
 y-1-\log y
 \wle \frac{(y-1)^2}{2 (y \wedge 1)^2}.
\end{equation}
\end{lemma}

\begin{proof}
Fix a number $x > -1$.  Define $f(t) = \log(1+tx)$ for $t \in [0,1]$.
Note that $f'(t) = x (1+tx)^{-1}$ and $f''(t) = - x^2 (1+tx)^{-2}$.
Then the formula
$f(1) = f(0) + \int_0^1 f'(r) \, dr $
implies that
\[
 \log(1+x) \weq \int_0^1 \frac{x}{1+rx} \, dr.
\]
Then \eqref{eq:Log1a} follows by noting that that $\frac{x}{1+x} \le \frac{x}{1+rx} \le x$ for all $0 \le r \le 1$.
Similarly, the formula $f(1) = f(0) + f'(0) + \int_0^1 \int_0^s f''(r) \, dr \, ds$ implies that
\[
 \log(1+x) \weq x - \int_0^1 \int_0^s \frac{x^2}{(1+rx)^2} \, dr \, ds.
\]
Then \eqref{eq:Log1} follows by noting that
$1 - x_- \le 1+rx \le 1 + x_+$ for all $0 \le r \le 1$.

Fix a number $y>0$. By substituting $x=y-1$ into \eqref{eq:Log1a}, we see that
$\frac{y-1}{y} \le \log y \le y-1$.
Hence $- \frac{\abs{y-1}}{y} \le \log y \le \abs{y-1}$, and \eqref{eq:Log2a} follows.
By substituting $x=y-1$ into \eqref{eq:Log1}, we see that
\[
 \frac{(y-1)^2}{2(1+(y-1)_+)^2}
 \wle y - 1 - \log y
 \wle \frac{(y-1)^2}{2(1-(y-1)_-)^2}.
\]
Now \eqref{eq:Log2} follows by noting that $1-(y-1)_- = y \wedge 1$.
\end{proof}

\begin{lemma}
\label{the:SquareRootForBrown}
If $t \ge 0$ satisfies $\abs{t-1} \ge c$ for some $c > 0$, then
$t+1 \le C (\sqrt{t}-1)^2$ where
\[
 C
 \weq 
 \begin{cases}
   \frac{2-c}{(\sqrt{1-c}-1)^2}, &\quad 0 < c \le 1, \\
   \frac{2+c}{(\sqrt{1+c}-1)^2}, &\quad c > 1.
 \end{cases}
\]
\end{lemma}
\begin{proof}
Differentiation shows that $r(t) = \frac{(\sqrt{t} - 1)^2}{t+1} = 1 - \frac{2\sqrt{t}}{t+1}$
is strictly decreasing on $[0,1]$ and strictly increasing on $[1,\infty)$.

(i) Assume that $0<c \le 1$. Then $\abs{t-1} \ge c$ implies that either $t \le 1-c$ or $t \ge 1+c$.
In the former case $r(t) \ge r(1-c)$, and in the latter case $r(t) \ge r(1+c)$.
Hence $t+1 \le ( r(1-c)^{-1} \wedge r(1+c)^{-1}) (\sqrt{t}-1)^2 = r(1-c)^{-1} (\sqrt{t}-1)^2$.

(ii) Assume that $c > 1$. Then $\abs{t-1} \ge c$ implies that $t \ge 1+c$, so that $r(t) \ge r(1+c)$.
Hence $t+1 \le r(1+c)^{-1} (\sqrt{t}-1)^2$.
\end{proof}

\begin{lemma}
\label{the:SquareRootKey}
For all $0 \le t \le c$,
\begin{equation}
 \label{eq:A1}
 \abs{t-1} \wle (1+c^{1/2}) \abs{\sqrt{t}-1}.
\end{equation}
For all $0 \le t \le c^{-1}$ and all $t \ge c$ with $c > 1$,
\begin{equation}
 \label{eq:SquareRootDiff}
 t+1 \wle \frac{c+1}{(c^{1/2} - 1)^2} (\sqrt{t} - 1)^2.
\end{equation}
\end{lemma}
\begin{proof}
(i) Because $\abs{t-1} = \abs{\sqrt{t}+1} \abs{\sqrt{t}-1}$,
we see that \eqref{eq:A1} follows by noting that
$\abs{\sqrt{t}+1} \le 1 + c^{1/2}$ for all $0 \le t \le c$.

(ii) Differentiation shows that $r(t) = \frac{(\sqrt{t} - 1)^2}{t+1} = 1 - \frac{2\sqrt{t}}{t+1}$
is decreasing on $[0,1]$ and increasing on $[1,\infty)$.
Hence $r(t) \ge r(c)$ for $t \ge c$ and
$r(t) \ge r(c^{-1})$ for $0 \le t \le c^{-1}$.
Because $r(c) = r(c^{-1}) = \frac{(c^{1/2} - 1)^2}{c+1}$,
we conclude \eqref{eq:SquareRootDiff}.
\end{proof}

\begin{lemma}
For all $t$ such that $\abs{\log t} \le 1$, 
\begin{align}
 \label{eq:A2}
 \abs{\log t +1 - t} &\wle 2 e^3 (\sqrt{t}-1)^2, \\
 \label{eq:A3}
 \log^2 t &\wle 4 e^3 (\sqrt{t}-1)^2.
\end{align}
\end{lemma}

\begin{proof}
Fix a number $t$ such that $\abs{\log t} \le 1$.  Then $e^{-1} \le t \le e$.
Then \eqref{eq:Log2} implies that 
$
\abs{\log t + 1 - t}
 \le \frac{(t-1)^2}{2 (t \wedge 1)^2}
 \le \frac12 e^2 (t-1)^2.
$
By combining this with \eqref{eq:A1}, inequality
\eqref{eq:A2} follows.  Furthermore, \eqref{eq:Log2a} implies that
$\abs{\log t} \le \frac{\abs{t-1}}{t\wedge 1}
\le e \abs{t-1}$.
By combining this with \eqref{eq:A1}, we that \eqref{eq:A3} is valid.
\end{proof}

\ifarxiv
  \bibliographystyle{alpha}
\else
  \bibliographystyle{IEEEtran}
\fi
\bibliography{lslReferences}

\ifleftovers\else
 \end{document}
\fi
